\pdfoutput=1
\documentclass[paper=a4, english, final ]{scrartcl}

\usepackage{babel}

\usepackage[T1]{fontenc}
\usepackage[utf8]{inputenc}

\usepackage{lmodern}

\usepackage[final, babel ]{microtype}

\usepackage{amsfonts} \usepackage{amssymb}  \usepackage{mathtools} 

\providecommand\given{} \newcommand\SetSymbol[1][]{
   \mathrel{}\mathclose{}#1|\allowbreak\mathopen{}\mathrel{}}
\DeclarePairedDelimiterX\Set[1]{\lbrace}{\rbrace}{ \renewcommand\given{\SetSymbol[\delimsize]} #1 }

\DeclarePairedDelimiter\abs{\lvert}{\rvert}\DeclarePairedDelimiter\norm{\lVert}{\rVert}

\makeatletter
\let\oldabs\abs
\def\abs{\@ifstar{\oldabs}{\oldabs*}}
\let\oldnorm\norm
\def\norm{\@ifstar{\oldnorm}{\oldnorm*}}

\let\oldSet\Set
\def\Set{\@ifstar{\oldSet}{\oldSet*}}
\makeatother

 \mathtoolsset{showonlyrefs}

\usepackage{chngcntr}

\usepackage{amsthm} \theoremstyle{plain}

\newtheorem{prop}[equation]{Proposition}

\newtheorem*{thm*}{Theorem}
\newtheorem*{prop*}{Proposition}

\newtheorem*{principle*}{Principle}

\theoremstyle{definition}

\newtheorem{lemma}[equation]{Lemma}
\newtheorem{defn}[equation]{Definition}

\newtheorem*{cor*}{Corollary}
\newtheorem*{lemma*}{Lemma}
\newtheorem*{defn*}{Definition}

\theoremstyle{remark}
\newtheorem{rem}[equation]{Remark}
\newtheorem{ex}[equation]{Example}

\newtheorem*{rem*}{Remark}
\newtheorem*{ex*}{Example} \usepackage{thmtools, thm-restate} 

\usepackage{tikz-cd} \tikzset{
	pf/.style={commutative diagrams/.cd, every arrow, every label},
	surj/.style=commutative diagrams/two heads,
	inj/.style=commutative diagrams/hook,
	gl/.style=commutative diagrams/equal,
	mat/.style={matrix of math nodes, commutative diagrams/.cd, every cell},
	dr/.style={matrix of math nodes, commutative diagrams/.cd, every cell, column sep=small},
	seq/.style={matrix of math nodes, commutative diagrams/.cd, every cell, column sep=small}
	}
\newenvironment{diag*}{\[\begin{tikzpicture}[commutative diagrams/.cd, every diagram, baseline=(current bounding box.center)]}{\end{tikzpicture}\]\ignorespacesafterend}
\newenvironment{diag}{\begin{equation}\begin{tikzpicture}[commutative diagrams/.cd, every diagram, baseline=(current bounding box.center)]}{\end{tikzpicture}\end{equation}\ignorespacesafterend}

\usepackage{csquotes} \usepackage[style=authoryear, date=year,
	isbn=false,
	]{biblatex}
\DeclareSourcemap{
	\maps[datatype=bibtex]{
		\map[overwrite]{
			\step[fieldset=pagetotal, null]
			\step[fieldset=translator, null]
			\step[fieldset=origlanguage, null]
			\step[fieldset=eventtitle, null]
			\step[fieldset=eventdate, null]
			\step[fieldset=venue, null]
			\step[fieldset=edition, null]
			\step[fieldset=pubstate, null]
		}
	}
}

\author{\texorpdfstring{Enno Keßler \and Artan Sheshmani \and Shing-Tung Yau}{
Enno Keßler, Artan Sheshmani, Shing-Tung Yau
}
}
\title{Torus Actions on Moduli Spaces of Super Stable Maps of Genus Zero}
\date{}

\usepackage{enumitem} \setlist[enumerate,1]{label = (\roman*)}
\usepackage{faktor}
\usepackage{slashed}
\usepackage{bm}

\usepackage{stensor}
\DeclareRobustCommand{\tensor}{\stensor}

\usepackage[final, pdfusetitle ]{hyperref}
\hypersetup{colorlinks=false}

\DeclareMathOperator{\ACI}{I}

\DeclareMathOperator{\Aut}{Aut}

\DeclareMathOperator{\Coker}{Coker}

\DeclareMathOperator{\GL}{GL}

\DeclareMathOperator{\im}{im}
\DeclareMathOperator{\id}{id}
\DeclareMathOperator{\ind}{ind}

\DeclareMathOperator{\PGL}{PGL}

\DeclareMathOperator{\SpGL}{Sp}
\DeclareMathOperator{\SGL}{SL}

\DeclareMathOperator{\Tr}{Tr}
\DeclareMathOperator{\UGL}{U}

\newcommand{\dual}[1]{{#1}^{\vee}}

\newcommand{\cat}[1]{\mathsf{#1}}

\newcommand{\Top}[1]{{\|#1\|}}
\newcommand{\Red}[1]{{#1}_{red}}
\newcommand{\Smooth}[1]{{|#1|}}

\renewcommand{\d}{\mathop{}\!d}

\newcommand{\VSec}[2][]{\Gamma_{#1}\left(#2\right)}

\newcommand{\tangent}[2][]{T_{#1}#2} \newcommand{\differential}[1]{\d{#1}} \newcommand{\cotangent}[1]{\dual{T}#1}

\newcommand{\Integers}{\mathbb{Z}}
\newcommand{\Z}{\Integers}

\newcommand{\RealNumbers}{\mathbb{R}}
\newcommand{\R}{\RealNumbers}
\newcommand{\ComplexNumbers}{\mathbb{C}}
\newcommand{\C}{\ComplexNumbers}

\newcommand{\ProjectiveSpace}[2][]{\mathbb{P}_{#1}^{#2}}

\newcommand{\cD}{\mathcal{D}}
\newcommand{\cO}{\mathcal{O}}

\newcommand{\targetACI}{J}
\newcommand{\DJBar}{\overline{D}_\targetACI}
\newcommand{\DelJBar}{\overline{\partial}_\targetACI}
\DeclareMathOperator{\rk}{rk}
\DeclareMathOperator{\DLaplace}{\Delta^\cD}
\DeclareMathOperator{\nablabar}{\overline{\nabla}}
\newcommand{\Dirac}{\slashed{D}}

\newcommand{\ev}{\mathrm{ev}}

\newcommand{\Fieldspace}{\mathcal{H}}
\newcommand{\Targetbundle}{\mathcal{E}}
\DeclareMathOperator{\DefBundle}{Def}
\DeclareMathOperator{\ObsBundle}{Obs}
\DeclareMathOperator{\SUSY}{SUSY}
\newcommand{\SCP}{\operatorname{SC} (\ProjectiveSpace[\C]{1|1})}
\newcommand{\Split}{\operatorname{Split}}

\newcommand{\NilIdeal}{\mathcal{I}_{nil}}

\begin{document}
\maketitle
\begin{abstract}
	We construct smooth \(\C^*\)-actions on the moduli spaces of super \(\targetACI\)-holo\-morph\-ic curves as well as super stable curves and super stable maps of genus zero and fixed tree type such that their reduced spaces are torus invariant.
	Furthermore, we give explicit descriptions of the normal bundles to the fixed loci in terms of spinor bundles and their sections.
	Main steps to the construction of the \(\C^*\)-action are the proof that the charts of the moduli space of super \(\targetACI\)-holomorphic curves obtained by the implicit function theorem yield a smooth split atlas and a detailed study of the superconformal automorphism group of \(\ProjectiveSpace[\C]{1|1}\) and its action on component fields.
\end{abstract}

\counterwithin{equation}{subsection}
\section{Introduction}

In~\cites{KSY-SJC}{KSY-SQCI} we have generalized the notions of \(\targetACI\)-holomorphic curves and stable maps to cases where the domain is a super Riemann surface and the target a classical almost Kähler manifold.
Super Riemann surfaces are supergeometric extensions of Riemann surfaces with spin structure having in addition a complex dimension with anti-commuting coordinates.
Their moduli spaces are superspaces which have their non-super counterpart as reduced space.
In this article we improve the understanding of how the super moduli spaces extend the classical counterpart by constructing torus actions that leave the reduced spaces invariant and describing the normal bundles to the inclusions.
The normal bundles are vector bundles on the non-super moduli spaces and are canonically described in terms of sections, twists and pullbacks of spinor bundles on the curves classified by the non-super moduli space.
In contrast, the torus actions are not canonical but rather are constructed by choosing explicit smooth split models of the super moduli spaces.
Nevertheless, the torus actions are constructed such that the standard forgetful maps and gluing maps are equivariant.

We are interested in torus actions on super moduli spaces that leave their reduced spaces invariant because the construction of an extension of Gromov--Witten invariants to the super case in the forthcoming paper~\cite{KSY-SGWIvTL} is motivated by torus localization.

Recall that a map \(\Phi\colon M\to N\) from a super Riemann surface \(M\) to an almost Kähler manifold \(N\) with symplectic form \(\omega\) and compatible almost complex structure \(\targetACI\) is called a super \(\targetACI\)-holomorphic curve if \(\Phi\) satisfies the following equation
\begin{equation}
	\DJBar \Phi
	= \frac12\left.\left(\differential{\Phi} + \ACI\otimes \targetACI \differential{\Phi}\right)\right|_\cD
	= 0.
\end{equation}
Here \(\ACI\) is the almost complex structure on the super Riemann surface and \(\cD\subset\tangent{M}\) the distribution defining the super Riemann surface structure.
In~\cite{KSY-SJC} we have constructed under certain conditions on the fixed super Riemann surface \(M\) and almost Kähler manifold \(N\) a moduli space \(\mathcal{M}(A)\) of maps \(\Phi\colon M\to N\) such that the homology class of the image of the reduced map satisfies \([\im \Red{\Phi}]=A\) for some fixed homology class \(A\in H_2(N)\).
The moduli space \(\mathcal{M}(A)\) is a real supermanifold of dimension
\begin{equation}
	2n\left(1-p\right) + 2\left< c_1(\tangent{N}), A\right>|2\left<c_1(\tangent{N}), A\right>,
\end{equation}
where \(2n\) is the real dimension of \(N\) and \(p\) is the genus of \(\Red{M}\).

For a superpoint \(C=\R^{0|s}\), the \(C\)-points of \(\mathcal{M}(A)\) are pairs \((\varphi, \psi)\) consisting of a map \(\varphi\colon \Red{M}\times C\to N\) satisfying an inhomogeneous \(\DelJBar\)-equation and a section \(\psi\in\VSec{\dual{S}\otimes_\C\varphi^*\tangent{N}}\) of the twisted spinor bundle \(\dual{S}\otimes_\C\varphi^*\tangent{N}\) satisfying an inhomogeneous Dirac-equation, see~\cite{KSY-SJC}.
It follows that the reduced manifold of \(\mathcal{M}(A)\) is the moduli space \(M(A)\) of \(\targetACI\)-holomorphic curves \(\phi\colon \Red{M}\to N\).
The fibers above \(\phi\) of the normal bundle \(N_{\mathcal{M}(A)/M(A)}\) of the embedding \(M(A)\hookrightarrow\mathcal{M}(A)\) above \(\phi\) are given by the holomorphic sections \(H^0(\dual{S}\otimes_\C \phi^*\tangent{N})\) of a twisted spinor bundle over the Riemann surface \(\Red{M}\).
The first main result of this paper is:
\begin{restatable*}{thm}{ModuliSpaceIsSplit}\label{thm:ModuliSpaceIsSplit}
	Assume that
	\begin{itemize}
		\item
			the super Riemann surface \(M\) is holomorphically relatively split, that is, the gravitino vanishes,
		\item
			the moduli space \(\mathcal{M}(A)\) is smoothly obstructed, and
		\item
			the target \(N\) is Kähler.
	\end{itemize}
	Then the moduli space \(\mathcal{M}(A)\) inherits a split atlas from \(\hat{\mathcal{H}}\).
\end{restatable*}

The split atlas constructed in~\ref{thm:ModuliSpaceIsSplit} yields an explicit identification of the sheaf of functions \(\cO_{\mathcal{M}(A)}\) with the sheaf of sections of the exterior algebra \(\bigwedge \dual{N}_{\mathcal{M}(A)/M(A)}\).
The proof of Theorem~\ref{thm:ModuliSpaceIsSplit} proceeds by showing that \(\mathcal{M}(A)\subset \hat{\mathcal{H}}\subset\mathcal{H}\) where \(\Fieldspace\) is the infinite-dimensional supermanifold of all maps \(M\to N\) for which we have constructed charts via the exponential map and component field decomposition in~\cite{KSY-SJC}.
\(\hat{\mathcal{H}}\) is an infinite-dimensional real subsupermanifold of the supermanifold~\(\mathcal{H}\) such that the charts induced from \(\Fieldspace\) yield a split atlas.
The inclusion \(\mathcal{M}(A)\subset\hat{\mathcal{H}}\) is then shown to respect the split atlas of \(\hat{\mathcal{H}}\).
Furthermore, \(\mathcal{M}(A)\) is equipped with a natural almost complex structure which is induced from the almost complex structure \(\targetACI\) on \(N\).
The particular split atlas of \(\mathcal{M}(A)\) and the almost complex structure on \(\mathcal{M}(A)\) allow to show that the map \((\varphi, \psi)\mapsto (\varphi, t\psi)\) for \(t\in\C^*\) extends to a smooth complex linear torus action \(\C^*\times \mathcal{M}(A)\to \mathcal{M}(A)\).

It is well known that the moduli space \(M(A)\) is in general not compact but has a natural compactification by stable maps.
In~\cite{KSY-SQCI} we have shown that in the case, where the domain is the only super Riemann surface of genus zero \(M=\ProjectiveSpace[\C]{1|1}\) a moduli space of super stable maps can be constructed as a functor
\begin{equation}
	\begin{split}
		\underline{\overline{\mathcal{M}}_{0,k}(A)}\colon \cat{SPoint}^{op}&\to \cat{Top} \\
		C &\mapsto \bigcup_{\text{tree types }T} \bigcup_{\Set{A_\alpha}}\underline{\mathcal{M}_{T}(\Set{A_\alpha})}(C)
	\end{split}
\end{equation}
from the category of superpoints, that is real supermanifolds of dimension \(0|s\), to the category of topological spaces.
The restriction \(\underline{\mathcal{M}_{T}(\Set{A_\alpha})}\) of \(\underline{\overline{\mathcal{M}}_{0,k}(A)}\) to super stable maps of fixed tree type \(T\) and distribution \(\Set{A_\alpha}\) of homology to the vertices of \(T\) is the point functor of a quotient superorbifold.
The topology on \(\underline{\overline{\mathcal{M}}_{0,k}(A)}\) is given by a generalization of Gromov topology.
The space \(\underline{\overline{\mathcal{M}}_{0,k}(A)}(\R^{0|0})\) of reduced points coincides with the moduli space \(\overline{M}_{0,k}(A)\) of classical stable maps of genus zero with \(k\) marked points as a topological space and is hence compact.

The moduli spaces \(\mathcal{M}_T(\Set{A_\alpha})\) are obtained as quotients of open subsets of products of several copies of \(\ProjectiveSpace[\C]{1|1}\) and \(\mathcal{M}(A)\) by superconformal automorphisms of \(\ProjectiveSpace[\C]{1|1}\).
Yet, the torus action on \(\ProjectiveSpace[\C]{1|1}\) and \(\mathcal{M}(A)\) does not directly descend to a torus action on \(\mathcal{M}_T(\Set{A_\alpha})\).
Instead we prove an equivalent description of the moduli spaces \(\mathcal{M}_T(\Set{A_\alpha})\) using successive quotients, obtaining a quotient of a split supermanifold by a subgroup of the superconformal automorphisms which commute with the torus action and is compatible with gluings.
This is the second main result of the paper:
\begin{restatable*}{thm}{TorusActionMTA}\label{thm:TorusActionMTA}
	Assume that \(N\) is Kähler, \(T\) is a \(k\)-marked tree and \(\Set{A_\alpha}\) be a partition of the homology class \(A\in H_2(N)\) such that the moduli space \(\mathcal{M}_T(\Set{A_\alpha})\) can be constructed as a quotient superorbifold.
	Then the superorbifold \(\mathcal{M}_T(\Set{A_\alpha})\) possesses a torus action such that its fixed points are precisely its \(\R^{0|0}\)-points.

	Furthermore, assume that the tree \(T\) arises from the \(k_1+1\) marked tree \(T_1\) and the \(k_2+1\)-marked tree \(T_2\) by connecting the vertices with labels \(k_1+1\) and \(k_2+1\) by an edge.
	Denote by \(\Set{A_{1\alpha}}\) and \(\Set{A_{2\alpha}}\) the restrictions of \(\Set{A_\alpha}\) to \(T_1\) and \(T_2\) respectively.
	Then the moduli spaces \(\mathcal{M}_{T_1}(\Set{A_{1\alpha}})\) and \(\mathcal{M}_{T_2}(\Set{A_{2\alpha}})\) carry torus actions such that their fixed points are precisely their reduced points and the gluing map
	\begin{equation}
		gl\colon \mathcal{M}_{T_1}(\Set{A_{1\alpha}})\times_N \mathcal{M}_{T_2}(\Set{A_{2\alpha}}) \to \mathcal{M}_T(\Set{A_\alpha})
	\end{equation}
	is equivariant.
\end{restatable*}

As the torus actions on \(\mathcal{M}_T(\Set{A_\alpha})\) are obtained via quotients of split supermanifolds, we also obtain a description of the normal bundles of the fixed loci of the torus action in terms of short exact sequences of vector bundles.
The terms in the short exact sequence are explained geometrically as parts of supersymmetry transformations and are expressed in terms of spinor bundles, their sections and restrictions to marked points.

We expect that the torus actions on \(\mathcal{M}_T(\Set{A_\alpha})\) yield a continuous torus action \(\C^*\times \underline{\overline{\mathcal{M}}_{0,k}(A)}\to \underline{\overline{\mathcal{M}}_{0,k}(A)}\) leaving the reduced manifold \(\overline{M}_{0,k}(A)\) invariant.
The proof is left for further work.

The outline of the paper is as follows:
In Section~\ref{Sec:SplitnessOfModuliSpace} we show that the charts of moduli space of super \(\targetACI\)-holomorphic maps from a holomorphically split super Riemann surface to a Kähler manifold constructed in~\cite{KSY-SQCI} yields a split atlas.
Section~\ref{Sec:AutomorphismsOfPC11} specializes to the case where the domain is of genus zero and describes the action of the superconformal automorphisms of \(\ProjectiveSpace[\C]{1|1}\) on super \(\targetACI\)-holomorphic curve.
In particular, we show that every superconformal automorphisms of \(\ProjectiveSpace[\C]{1|1}\) can be written as a product of an element of \(\SGL_{\C}(2)\) and a supersymmetry transformation.
In Section~\ref{Sec:SplitnessStableMapsModuli} we use the splitness of the moduli spaces to define a torus action on the moduli space of super \(\targetACI\)-holomorphic curves that has classical \(\targetACI\)-holomorphic maps as fixed points.
Furthermore we extend the torus action to the moduli space of all super stable maps of genus zero.

\subsection*{Acknowledgments}
Enno Keßler and Artan Sheshmani thank the Simons Center for Geometry and Physics at Stony Brook for the invitation to the program on SuperGeometry and SuperModuli where this work was presented.
We thank Alexander Polishchuk for useful discussions.
 
\section{Splitness of the moduli space of super \texorpdfstring{\(J\)}{J}-holomorphic curves}\label{Sec:SplitnessOfModuliSpace}

In this section we recall the construction of charts for the moduli space \(\mathcal{M}(A)\) from a fixed super Riemann surface \(M\) to a fixed target \(N\) and show that this moduli space inherits a split structure from the space of maps under certain conditions on the domain and target.
Furthermore, we construct almost complex structures on the deformation and obstruction bundles.

In Section~\ref{SSec:SplitSupermanifolds} we recall the general notion of split supermanifold also in the setting of infinite-dimensional supermanifolds.
Section~\ref{SSec:DeformationsObstructions} recalls the definition of super \(\targetACI\)-holomorphic curve from a super Riemann surface and defines when its moduli space is smoothly obstructed.
The moduli space of super \(\targetACI\)-holomorphic curves is, if it is smoothly obstructed, a smooth subsupermanifold of a split subsupermanifold \(\hat{\Fieldspace}\) of the supermanifold of all maps \(M\to  N\), that is \(\mathcal{M}(A)\subset\hat{\Fieldspace}\subset\Fieldspace\).
We discuss a split atlas for \(\hat{\Fieldspace}\) in Section~\ref{SSec:SpaceOfMaps} and the induced split atlas of \(\mathcal{M}(A)\) in~Section~\ref{SSec:TheModuliSpace}.
In Section~\ref{SSec:ComplexStructures} we show that the almost complex structure \(\targetACI\) of the target \(N\) induces almost complex structures on \(\tangent{\Fieldspace}\), \(\tangent{\hat{\Fieldspace}}\) as well as the deformation and obstruction bundle on \(\mathcal{M}(A)\).

\subsection{Split Supermanifolds}\label{SSec:SplitSupermanifolds}
Supermanifolds of dimension \(m|n\) can be defined as ringed spaces \(M=(\Top{M}, \cO_M)\) which are locally isomorphic to \((\R^n, \cO_{\R^{m|n}}=\mathcal{C}^\infty(\R^m, \R)\otimes \bigwedge_n)\).
Here \(\Top{M}\) is a second countable, Hausdorff (euclidean) topological space and \(\cO_M\) is a sheaf of super rings.
The sheaf of super rings~\(\cO_{\R^{m|n}}\) is obtained as the tensor product of the sheaf of smooth \(\R\)-valued functions on \(\R^{m}\) with the real Graßmann algebra \(\bigwedge_n\) of \(n\) generators.
For more details on supergeometry in the ringed space formalism, we refer to the early overview~\cite{L-ITS}, the very concise~\cite{DM-SUSY}, as well as the first part of the textbook~\cite{EK-SGSRSSCA}.

Let \(\NilIdeal\subset \cO_M\) be the sheaf of ideals of nilpotent elements.
The ringed space \(\Red{M}=(\Top{M}, \faktor{\cO_M}{\NilIdeal})\) is a smooth manifold whose topological space coincides with the one of \(M\).
Furthermore, there is a canonical map of supermanifolds \(\Red{i}\colon \Red{M}\to M\) induced by the map \(\cO_M\to \faktor{\cO_M}{\NilIdeal}\) which is the identity on the topological spaces.
We obtain the short exact sequence
\begin{diag}\matrix[mat](m) {
		0 & \tangent{\Red{M}} & \Red{i}^*\tangent{M} & N_{M/\Red{M}} & 0\\
	} ;
	\path[pf] {
		(m-1-1) edge (m-1-2)
		(m-1-2) edge node[auto]{\(\differential{\Red{i}}\)} (m-1-3)
		(m-1-3) edge (m-1-4)
		(m-1-4) edge (m-1-5)
	};
\end{diag}
Here \(\tangent{\Red{M}}\) is a super vector bundle of rank \(m|0\) and the normal bundle \(N_{M/\Red{M}}\) is a super vector bundle of rank \(0|n\) over \(\Red{M}\).
The sheaf of sections of \(N_{M/\Smooth{M}}\) is given by \(\dual{\left(\faktor{\NilIdeal}{\NilIdeal^2}\right)}\).

Conversely, one can construct a supermanifold out of a vector bundle:
For a vector bundle \(E\to \Smooth{M}\) of rank \(n\) over a smooth manifold \(\Smooth{M}\) of dimension \(m\), the split supermanifold
\begin{equation}
	\Split E= \left(\Smooth{M}, \VSec{\bigwedge(\dual{E})}\right)
\end{equation}
constructed from \(E\) is a supermanifold of dimension \(m|n\).
Split supermanifolds have a coordinate atlas where the gluing functions are particularly simple.
Suppose that the vector bundle \(E\) is trivialized over the coordinate neighborhoods \(V_i\subset\Top{M}\) and denote the coordinates on \(V_i\) by \(y_j^a\) and the real framing of \(E\) over \(V_i\) by \(e_{i,\alpha}\).
If \(V_i\cap V_j\) is non-empty we have coordinate change functions \(f_{ij}\) and a linear map \(F_{ij}\) between the framings
\begin{align}
	y_j^a &= f_{ij}^a(y_i), &
	e_{j,\alpha} &= \tensor[_\alpha]{F}{_j_i^\beta}(x_i)e_{i,\beta}.
\end{align}
The split supermanifold \(\Split E\) then has the atlas \(\Set{U_i}\) with local coordinates \((x_i^a, \eta_i^\alpha)\) with coordinate changes
\begin{align}\label{eq:GluingSplitSuperManifold}
	x_j^a &= f_{ij}^a(x_i), &
	\eta_j^\alpha &= \eta_i^\beta \tensor[_\beta]{F}{_i_j^\alpha}(x_i).
\end{align}
In contrast to arbitrary gluing functions of supermanifolds, the gluing functions for the even coordinates \(x_i^a\) do not depend on \(\eta_i^\alpha\) and \(\eta_j^\alpha\) depends linearly on \(\eta_i^\beta\).
The normal bundle \(N_{\Split E/\Smooth{M}}\) is isomorphic to \(\Pi E\) where \(\Pi\) is the parity reversal functor on super vector bundles.

By Batchelor's Theorem, see~\cite{B-SOSM}, every smooth supermanifold \(M\) allows to construct splittings, that is smooth super diffeomorphisms \(s\colon \Split N_{\Smooth{M}/M} \to M\).
But importantly, there are different splittings and no canonical choice.
In this work, we will specify a splitting of a supermanifold \(M\) by constructing an atlas for \(M\) such that the gluing functions have the special form as in Equations~\eqref{eq:GluingSplitSuperManifold}.

Complex structures on supermanifolds can be specified equivalently via a complex atlas or via an almost complex structure on the tangent bundle of a smooth supermanifold, see~\cites{V-ACSOM}[Section~7]{EK-SGSRSSCA}.
The reduced space of an (almost) complex super manifold is an (almost) complex manifold and the normal bundle carries an almost complex structure.
The normal bundle is holomorphic if the reduced manifold is complex.
If \(E\to \Smooth{M}\) is a holomorphic vector bundle \(\Split E\) is a complex super manifold.
But Batchelor's theorem does not hold for complex supermanifolds.
That is, in general, one cannot find biholomorphic maps \(\Split N_{\Smooth{M}/M}\to M\) even though there exist diffeomorphisms \(\Split N_{\Smooth{M}/M}\to M\).
Obstructions to the existence of holomorphic splittings have been studied for example in~\cites{G-HGM}[Chapter~4, §2]{M-GFTCG}{DW-SMNP}.
Explicit examples of non-split holomorphic supermanifolds have been constructed in~\cites[Chapter~4, §2, 10.~Examples]{M-GFTCG}{CNR-NPCYSMP2}.

Besides the ringed space formalism we also need the functor of points approach to supermanifolds, see~\cites{M-IDCSM}{S-GAASTS} as well as the earlier~\cites{S-DSP}{V-MSM}.
The advantage of the Molotkov--Sachse approach to supermanifolds is to allow infinite-dimensional supermanifolds.
A supermanifold in the Molotkov--Sachse approach is a representable functor \(\underline{M}\colon \cat{SPoint^{op}}\to \cat{Man}\) from the opposite of the category of super points to a suitable category of manifolds.
Here the category of super points consists of the superpoints \(\R^{0|s}\) with morphisms between them.
The category \(\cat{SPoint^{op}}\) is equivalent to the category of finitely generated Graßmann algebras.

\begin{lemma}\label{lemma:NormalBundleFunctorOfPoints}
	Let \(M\) be a finite-dimensional supermanifold and \(\underline{M}\) its functor of points.
	Then \(\underline{M}(\R^{0|0})=\Red{M}\) and \(\underline{M}(\R^{0|1})=N_{M/\Red{M}}\) where the projection \(N_{M/\Red{M}}\to \Red{M}\) is given by \(\underline{M}(\R^{0|0}\to \R^{0|1})\).
\end{lemma}
\begin{proof}
	Suppose that the supermanifold \(M\) is given by an atlas \(\Set{U_i}\) such that the open subsupermanifold \(U_i\) is isomorphic to \(\R^{m|n}\) with coordinates \((x_i^a, \eta_i^\alpha)\) and coordinate changes \(f_{ij}\) such that
	\begin{align}
		x_j^a &= \tensor{F}{_i_j^a}(x_i, \eta_i) = \tensor[_0]{F}{_i_j^a}(x_i) + \eta_i^\mu \eta_i^\nu \tensor[_\nu_\mu]{F}{_i_j^a}(x_i) + \dotsb, \\
		\eta_j^\alpha &= \tensor{F}{_i_j^\alpha}(x_i, \eta_i) = \eta_i^\mu \tensor[_\mu]{F}{_i_j^\alpha}(x_i) + \dotsb.
	\end{align}
	The reduced manifold \(\Red{M}\) is given by the atlas \(\Set{V_i}\) such that \(V_i\) and \(U_i\) have the same topological space but \(V_i\) is isomorphic to \(\R^m\) with coordinates \((y_i^a)\) and coordinate changes \(y_j^a= \tensor[_0]{F}{_i_j^a}(y_i)\).
	The inclusion \(i\colon \Red{M}\to M\) is given in the local coordinates by
	\begin{align}
		i^\# x_i^a &= y_i^a, &
		i^\# \eta_i^\alpha &= 0.
	\end{align}
	The tangent bundle \(\tangent{M}\) has frames \(\partial_{x_i^a}, \partial_{\eta_i^\alpha}\) such that
	\begin{equation}
		\begin{pmatrix}
			\partial_{x_i^a} \\
			\partial_{\eta_i^\alpha}\\
		\end{pmatrix}
		=
		\begin{pmatrix}
			\frac{\partial \tensor{F}{_i_j^b}(x_i, \eta_i)}{\partial_{x_i^a}} & \frac{\partial \tensor{F}{_i_j^\beta}(x_i, \eta_i)}{\partial_{x_i^\alpha}} \\
			\frac{\partial \tensor{F}{_i_j^b}(x_i, \eta_i)}{\partial_{\eta_i^a}} & \frac{\partial \tensor{F}{_i_i^\beta}(x_i, \eta_i)}{\partial_{\eta_i^\alpha}}
		\end{pmatrix}
		\begin{pmatrix}
			\partial_{x_j^b} \\
			\partial_{\eta_j^\beta}\\
		\end{pmatrix}.
	\end{equation}
	Consequently, the normal bundle has frames \(n_{i,\alpha} = i^*\partial_{\eta_i^\alpha}\) over \(V_i\) such that \(n_{i,\alpha} = \tensor[_\alpha]{F}{_i_j^\beta}(x_i) n_{j,\beta}\).

	Let now \(p\colon \R^{0|0}\to M\) be an \(\R^{0|0}\) point of \(M\) which is given in the coordinates \((x_i^a, \eta_i^\alpha)\) by
	\begin{align}
		p^\# x_i^a &= p_i^a, &
		p^\# \eta_i^\alpha &= 0.
	\end{align}
	That is, the set of \(\R^{0|0}\) points in \(U_i\) is given by the point \((p_i^a)\) in \(\R^{m}\).
	The change of coordinates
	\begin{equation}
		p_i^a
		= p^\# x_i^a
		= p^\# \tensor{F}{_j_i^a}(x_j, \eta_j)
		= \tensor[_0]{F}{_j_i^a}(p_j)
	\end{equation}
	implies that the manifold of \(\R^{0|0}\)-points coincides with the manifold \(\Red{M}\).

	Let \(\lambda\) be the odd coordinate of \(\R^{0|1}\) and \(p\colon \R^{0|1}\to M\) be an \(\R^{0|1}\)-point of \(M\).
	In the coordinates \((x_i^a, \eta_i^\alpha)\), the \(\R^{0|1}\)-point \(p\) is given by
	\begin{align}
		p^\# x_i^a &= p_i^a, &
		p^\# \eta_i^\alpha &= \lambda p_i^\alpha .
	\end{align}
	That is, the set of \(\R^{0|1}\)-points in \(U_i\) is given by the points \((p_i^a, p_i^\alpha)\) in \(\R^{m}\times \R^n\).
	The inclusion \(l\colon \R^{0|0}\to \R^{0|1}\) is given by \(l^\#\lambda=0\).
	Hence the \(\R^{0|1}\)-point \(p\) given by \((p_i^a, p_i^\alpha)\) is mapped under \(\underline{M}(\R^{0|0}\to \R^{0|1})\) to \(p\circ l\) given by \((p_i^a)\).

	We map the point \((p_i^a, p_i^\alpha)\) of \(\underline{U_i}(\R^{0|1})\) to the point \((p_i^a, p_i^\alpha n_{i,\alpha})\) in the trivialization of the normal bundle above \(V_i\).
	The claim then follows from the following verification of that the change of coordinates for \(p\) agree with the gluing functions of \(N_{M/\Red{M}}\):
	\begin{align}
		p_j^\alpha n_{j,\alpha}
		= p_i^\beta \tensor[_\beta]{F}{_i_j^\alpha} \tensor[_\alpha]{F}{_j_i^\gamma} n_{i_\gamma}
		= p_i^\beta n_{i,\beta}.
	\end{align}
\end{proof}

In light of Lemma~\ref{lemma:NormalBundleFunctorOfPoints}, we will say that \(\underline{M}(\R^{0|0})\) is the reduced manifold and \(\underline{M}(\R^{0|1})\) is the normal bundle of an infinite-dimensional supermanifold~\(M\).
If \(E\to \Smooth{M}\) is a vector bundle of possibly infinite rank over a classical possibly infinite-dimensional manifold~\(\Smooth{M}\), the split supermanifold constructed from \(E\) is given by an atlas such that all gluing functions are in a form as in~\eqref{eq:GluingSplitSuperManifold}.
More explicitly, assume that the vector bundle \(E\) is trivialized over the open sets \(V_i\subset{\Smooth{M}}\), has the vector space \(G\) as fibers and gluing functions
\begin{align}
	f_{ij}\colon V_i&\to V_j, &
	F_{ij}\colon V_i\times G &\to G.
\end{align}
Then the supermanifold \(\underline{\Split E}\) is constructed from local charts \(\underline{U_i}\) such that
\begin{equation}
	\underline{U_i}(B) = V_i\otimes {\left(\cO_B\right)}_0 \oplus G\otimes {\left(\cO_B\right)}_1
\end{equation}
and gluing maps
\begin{equation}
	f_{ij}\otimes {\left(\cO_B\right)}_0\oplus F_{ij}\otimes {\left(\cO_B\right)}_1\colon \underline{U_i}(B)\to \underline{U_j}(B).
\end{equation}

\begin{ex}[The holomorphically split supermanifold {\(\ProjectiveSpace[\C]{1|1}\)}]\label{ex:P11isSplit}
	Recall that the \(B\)-points of \(\ProjectiveSpace[\C]{1|1}\) classify \(B\)-linear maps \(\C^{1|0}\times B\to\C^{2|1}\).
	Such linear maps are given by tuples \([Z^0:Z^1:\Theta]\) for \(Z^i\in{\left(\cO_B\right)}_0\) with one of them invertible and \(\Theta\in{\left(\cO_B\right)}_1\).
	Two two such tuples are equivalent if they differ by an even invertible element \(\lambda\in{\left(\cO_B\right)}_0^*\), that is
	\begin{equation}
		[Z_1:Z_2:\Theta] \sim [\lambda Z_1: \lambda Z_2 : \lambda \Theta].
	\end{equation}
	Particularly important points of \(\ProjectiveSpace[\C]{1|1}\) are \(0=[0:1:0]\), \(1_\epsilon=[1:1:\epsilon]\) and \(\infty=[1:0:0]\).

	The super projective space \(\ProjectiveSpace[\C]{1|1}\) carries the structure of a complex supermanifold given by the two charts \(V_i=\C^{1|1}\), \(i=1,2\) with coordinates \((z_i, \theta_i)\) given by
	\begin{align}
		z_1 &= \frac{Z_1}{Z_2}, &
		\theta_1 &= \frac{\Theta}{Z_2}, &
		z_2 &= -\frac{Z_2}{Z_1}, &
		\theta_2 &= \frac{\Theta}{Z_1}.
	\end{align}
	The coordinate change between \((z_1, \theta_1)\) and \((z_2, \theta_2)\) is given by
	\begin{align}\label{eq:P11CoordinateChange}
		z_1 &= -\frac{1}{z_2}, &
		\theta_1 &= -\frac{\theta_2}{z_2}; &
		z_2 &= -\frac{1}{z_1}, &
		\theta_2 &= \frac{\theta_1}{z_1}.
	\end{align}
	We will explain the reason for the additional minus signs in Example~\ref{ex:SJConP11} below.

	There is a unique holomorphic map \(i\colon\ProjectiveSpace[\C]{1}\to \ProjectiveSpace[\C]{1|1}\) which is the identity on topological spaces.
	We will denote the coordinates on \(\ProjectiveSpace[\C]{1}\) likewise by \(z_1\), \(z_2\).
	Then the map \(i\) is given by
	\begin{align}
		i^\# z_1&= z_1, &
		i^\# \theta_1 &= 0, &
		i^\# z_2&= z_2, &
		i^\# \theta_2 &= 0.
	\end{align}
	By the equations of coordinate change we have
	\begin{align}
		\partial_{z_1} &= z_2^2\partial_{z_2} + z_2\theta_2\partial_{\theta_2} &
		\partial_{\theta_1} &= -z_2\partial_{\theta_2}
	\end{align}
	and hence \(i^*\partial_{z_1}=z_2^2i^*\partial_{z_2}\) and \(i^*\partial_{\theta_1} = -z_2i^*\partial_{\theta_2}\).
	That is, we have a short exact sequence of vector bundles over \(\ProjectiveSpace[\C]{1}\)
	\begin{diag}\matrix[mat](m) {
			0 & \tangent{\ProjectiveSpace[\C]{1}}=\cO(2) & i^*\tangent{\ProjectiveSpace[\C]{1|1}} & N_{\ProjectiveSpace[\C]{1|1}/\ProjectiveSpace[\C]{1}}=\cO(1) & 0 \\
		} ;
		\path[pf] {
			(m-1-1) edge (m-1-2)
			(m-1-2) edge node[auto]{\(\differential{i}\)} (m-1-3)
			(m-1-3) edge (m-1-4)
			(m-1-4) edge (m-1-5)
		};
	\end{diag}
	Here \(\cO(1)\) and \(\cO(2)\) are the canonical line bundles on \(\ProjectiveSpace[\C]{1}\) with two resp.\ three sections.
	Notice that \(\cO(1)\otimes \cO(1)=\cO(2)=\tangent{\ProjectiveSpace[\C]{1}}\).
	For this reason, we also denote \(S=\cO(1)\) the spinor bundle on \(\ProjectiveSpace[\C]{1}\), compare~\cite[Appendix~A]{EK-SGSRSSCA}.

	Hence, \(\ProjectiveSpace[\C]{1|1} = \Split S = \Split \cO(1)\).
\end{ex}

\subsection{Deformations and Obstructions}\label{SSec:DeformationsObstructions}
In~\cite{KSY-SJC} we have defined a super \(\targetACI\)-holomorphic curve to be a map \(\Phi\colon M\to N\) from a super Riemann surface \(M\) to an almost Kähler manifold \((N, \omega, \targetACI)\) satisfying
\begin{equation}\label{eq:DJBar}
	\DJBar\Phi
	= \frac12\left.\left(\differential{\Phi} + \ACI\otimes \targetACI \differential{\Phi}\right)\right|_\cD
	= 0.
\end{equation}
In this article we will mainly be concerned with the case \(M=\ProjectiveSpace[\C]{1|1}\), that is, the only super Riemann surface of genus zero.
However, for the moment it is sufficient to suppose that \(M\) is a super Riemann surface over \(B=\R^{0|0}\).
In that case, there is a holomorphic embedding \(i\colon \Smooth{M}\to M\) where \(\Smooth{M}\) is the reduced space of \(M\), a Riemann surface.
The super Riemann surface \(M\) is then completely determined by the holomorphic structure on \(\Smooth{M}\), or equivalently a conformal class of metrics, and the spinor bundle \(S=i^*\cD\).
With respect to this embedding we can decompose any map \(\Phi\colon M\times C\to N\) parametrized by a superpoint \(C=\R^{0|s}\) in the component fields
\begin{equation}\label{eq:ComponentFields}
	\begin{aligned}
		\varphi &= \Phi\circ i\colon \Smooth{M}\times C \to N, \\
		\psi &= \left.i^*\differential{\Phi}\right|_\cD \in \VSec{\dual{S}\otimes \varphi^*\tangent{N}}, \\
		F &= i^*\DLaplace \Phi \in\VSec{\varphi^*\tangent{N}},
	\end{aligned}
\end{equation}
see~\cite[Chapter~12.1]{EK-SGSRSSCA}.
The condition~\eqref{eq:DJBar} for \(\Phi\) to be a super \(\targetACI\)-holomorphic curve can be rewritten as
\begin{align}\label{eq:ComponentFieldEquations}
	\begin{aligned}
		0 &= \DelJBar\varphi + \frac14\Tr_{\dual{g}_S}\left(\gamma\otimes\mathfrak{j}\targetACI\right)\psi, &
		0 &= F, \\
		0 &= \left(1+\ACI\otimes\targetACI\right)\psi, &
		0 &= \Dirac\psi - \frac13SR^N(\psi),
	\end{aligned}
\end{align}
see~\cite[Corollary~2.5.3]{KSY-SJC}.
In particular, for \(C=\R^{0|0}\) the map \(\phi\colon \Smooth{M}\to N\) is a classical \(\targetACI\)-holomorphic curve and any \(\targetACI\)-holomorphic curve \(\phi\) induces a super \(\targetACI\)-holomorphic curve \(\Phi_C\colon M\times C\to N\) parametrized by \(C\) and with component fields \(\varphi=\phi\), \(\psi=0\) and \(F=0\).

\begin{ex}[Super \(\targetACI\)-holomorphic curves on {\(\ProjectiveSpace[\C]{1|1}\)}]\label{ex:SJConP11}
	The projective superspace \(\ProjectiveSpace[\C]{1|1}\) is the only super Riemann surface of genus zero, see~\cites{CR-SRSUTT}[Theorem~9.4.1]{EK-SGSRSSCA}.
	The distribution \(\cD\) of the super Riemann surface structure of \(\ProjectiveSpace[\C]{1|1}\) is generated in the superconformal coordinates from Example~\ref{ex:P11isSplit} by
	\begin{align}
		D_1 &= \partial_{\theta_1} + \theta_1\partial_{z_1}, &
		D_2 &= \partial_{\theta_2} + \theta_2\partial_{z_2}.
	\end{align}
	The coordinate change~\eqref{eq:P11CoordinateChange} implies \(D_1=-z_2D_2\).
	Note that the minus signs in~\eqref{eq:P11CoordinateChange} are crucial.
	The pullbacks \(s_i=i^*D_i\) of \(D_i\) along \(i\colon \ProjectiveSpace[\C]{1}\to \ProjectiveSpace[\C]{1|1}\) are local holomorphic frames for \(S\) over the coordinate neighborhoods \(V_i\).

	Let now \(\Phi\colon \ProjectiveSpace[\C]{1|1}\times B\to N\) be a super \(\targetACI\)-holomorphic curve into a Kähler manifold~\(N\).
	Its component fields \(\varphi\colon \ProjectiveSpace[\C]{1}\to N\), \(\psi\in\VSec{\dual{S}\otimes \varphi^*\tangent{N}}\) and \(F\in\VSec{\varphi^*\tangent{N}}\) defined in~\eqref{eq:ComponentFields} satisfy the Equations~\eqref{eq:ComponentFieldEquations}.
	As the target is Kähler this implies that \(\varphi\) is a holomorphic map, \(\psi\) a holomorphic section of \(\dual{S}\otimes_\C \varphi^*\tangent{N}\), see~\cite[Lemma~2.3.6]{KSY-SJC}.
	Picking local coordinates \(X^a\) on \(N\) we can write the map \(\Phi\) in local coordinates as
	\begin{equation}
		\Phi^\#X^a
		= \varphi^a(z_1) + \theta_1\psi_1^a(z_1)
		= \varphi^a(z_2) + \theta_2 \psi_2^a(z_2).
	\end{equation}
	Then the component fields have the form
	\begin{align}
		\varphi^\#X^a
		&= \varphi^a(z_1)
		= \varphi^a(z_2), &
		\psi
		&= \dual{s}_1\otimes \left(\psi_1^a(z_1) \varphi^*\partial_{X^a}\right)
		= \dual{s}_2\otimes \left(\psi_2^a(z_2) \varphi^*\partial_{X^a}\right),
	\end{align}
	where \(\dual{s}_i\) is the dual frame to \(s_i\).
\end{ex}

Under certain conditions, we have constructed the moduli space of super \(\targetACI\)-holomorphic curves as a subsupermanifold of the infinite-dimensional supermanifold  \(\Fieldspace\) of all maps \(\Phi\colon M\to N\).
The supermanifold \(\Fieldspace\) is constructed in the Molotkov--Sachse approach as a functor
\begin{equation}
	\underline{\Fieldspace}\colon \cat{SPoint}^{op}\to \cat{Man}
\end{equation}
from the category of superpoints to the category of (Fréchet-) manifolds.
Charts are constructed in~\cite[Section~3.2]{KSY-SJC} via the exponential map on the target \(N\).

The moduli space of super \(\targetACI\)-holomorphic curves is a subfunctor of \(\underline{\Fieldspace}\) defined by
\begin{equation}
	\underline{\mathcal{M}(A)}(C) = \Set{\Phi\colon M\times C\to N \given \DJBar\Phi=0, [\im \Red{\Phi}]=A}.
\end{equation}
Here \(A\in H_2(N, \Z)\) is the fixed homology class of the image.
Without further assumptions the functor \(\underline{\mathcal{M}(A)}\) takes values in the category of topological spaces:
\begin{equation}
	\underline{\mathcal{M}(A)}\colon \cat{SPoint}^{op}\to \cat{Top}.
\end{equation}

The moduli space \(\underline{\mathcal{M}(A)}\) is the zero locus of the operator \(\DJBar\) which is a section of a vector bundle \(\Targetbundle\to \Fieldspace\).
Here \(\Targetbundle\) is the vector bundle such that the fiber above \(\Phi\colon M\to N\) is given by
\begin{equation}
	\underline{\Targetbundle}_\Phi(C) = {\left({\VSec{\dual{\cD}\otimes\Phi^*\tangent{N}}}^{0,1}\otimes\cO_C\right)}_0
\end{equation}
In~\cite[Section~3.3]{KSY-SJC} we have shown that \(\Targetbundle\) is indeed a vector bundle and that \(\mathcal{S}=\DJBar\) is a section.
Here, we are interested in the kernel \(\ker \differential{\mathcal{S}}\) and cokernel \(\Coker \differential{\mathcal{S}}\) of its differential
\begin{equation}
	\underline{\differential{\mathcal{S}}}\colon \underline{\tangent{\Fieldspace}}\to \underline{\mathcal{S}^*\tangent{\Targetbundle}}.
\end{equation}
In general, the rank of \(\ker \differential{\mathcal{S}}\) and \(\Coker\differential{\mathcal{S}}\) is not constant, but might depend on the point \(\Phi\) in \(\underline{\Fieldspace}\).
In particular, \(\ker \differential{\mathcal{S}}\) and \(\Coker\differential{\mathcal{S}}\) are not super vector bundles.
However the ranks of \(\ker\differential{\mathcal{S}}\) and \(\Coker\differential{\mathcal{S}}\) are not completely unrelated.
In~\cite{KSY-SJC}, we have shown that for any super \(\targetACI\)-holomorphic curve \(\Phi_C\colon M\times C\to N\) with component fields \(\varphi=\phi\times C\), \(\psi=0\) and \(F=0\) we have
\begin{equation}
	\rk \ker\differential{\mathcal{S}} - \rk\Coker\differential{\mathcal{S}}
	= \ind D_\phi|\ind \Dirac^{1,0}
	= 2n(1-p) + 2\left<c_1(\tangent{N}), A\right>|2\left<c_1(\tangent{N}), A\right>.
\end{equation}
In particular, the index depends only on the homology class \(A=[\im\phi]\in H_2(N)\).

\begin{defn}\label{defn:SmoothlyObstructed}
	We say that the moduli space \(\mathcal{M}(A)\) is smoothly obstructed if
	\begin{itemize}
		\item
			\(\mathcal{M}(A)\) is a smooth subsupermanifold of \(\Fieldspace\) of real dimension \(2d_e|2d_0\) with embedding \(j\colon \mathcal{M}(A)\to \Fieldspace\), and
		\item
			the tangent bundle \(\tangent{\mathcal{M}(A)}\) is isomorphic to \(j^*\ker\differential{\mathcal{S}}\).
	\end{itemize}
\end{defn}
In this situation, we obtain by pullback along \(j\) an exact sequence of vector bundles
\begin{diag}\matrix[mat](m) {
		0 & j^*\ker\differential{\mathcal{S}} & j^*\tangent{\Fieldspace} & j^*\mathcal{S}^*\Targetbundle & j^*\Coker\differential{\mathcal{S}} & 0 \\
	} ;
	\path[pf] {
		(m-1-1) edge (m-1-2)
		(m-1-2) edge (m-1-3)
		(m-1-3) edge node[auto]{\(j^*\differential{\mathcal{S}}\)} (m-1-4)
		(m-1-4) edge (m-1-5)
		(m-1-5) edge (m-1-6)
	};
\end{diag}
We call  \(j^*\ker\differential{\mathcal{S}}\to\mathcal{M}(A)\) the deformation bundle \(\DefBundle\) and \(j^*\Coker\differential{\mathcal{S}}\to \mathcal{M}(A)\) the obstruction bundle \(\ObsBundle\).
A section of \(\DefBundle\) over \(\Phi_C\in\underline{\mathcal{M}(A)}(C)\) gives an infinitesimal deformation of the map \(\Phi_C\) that satisfies the equation of super \(\targetACI\)-holomorphic curves to first order.
The obstruction bundle measures how far \(\differential{\mathcal{S}}\) is not surjective.

In~\cite{KSY-SJC}, we have shown that there are targets \(N\) for which \(\differential{\mathcal{S}}\) is always surjective.
In that case the obstruction bundle is zero and the above assumptions are satisfied, that is, \(\mathcal{M}(A)\) is a subsupermanifold of \(\Fieldspace\).

\subsection{The space of maps}\label{SSec:SpaceOfMaps}
In this section we first recall the construction of the supermanifold \(\Fieldspace\) of all maps \(\Phi\colon M\to N\) from~\cite{KSY-SJC}.
We then show that there is a split supermanifold \(\hat{\Fieldspace}\) such that \(\mathcal{M}(A)\subset \hat{\Fieldspace}\subset\Fieldspace\).

Charts on \(\Fieldspace\) have been constructed in~\cite[Chapter~3.2]{KSY-SJC}.
Locally, around \(\Phi\colon M\to N\) given by \(\varphi=\phi\), \(\psi=0\), \(F=0\) the functor \(\underline{\Fieldspace}\) is isomorphic to an open subfunctor \(\underline{U}_\phi\) around zero of
\begin{equation}
	\begin{split}
		\underline{W_\phi} = \underline{\VSec{\Phi^*\tangent{N}}}\colon \cat{SPoint}^{op}&\to \cat{Man} \\
		C &\mapsto {\left(\VSec{\Phi^*\tangent{N}}\otimes\cO_C\right)}_0
	\end{split}
\end{equation}
The module \({\left(\VSec{\Phi^*\tangent{N}}\otimes\cO_C\right)}_0\) which appears on the right hand side is isomorphic to \({\VSec{\Phi_C^*\tangent{N}}}_0\), where \(\Phi_C\colon M\times C\to N\) is the base extension of \(\Phi\).
As explained in~\cite[Chapter~13.10]{EK-SGSRSSCA}, any section~\(X\in\underline{W_\phi}(C)={\VSec{\Phi_C^*\tangent{N}}}_0\) is completely determined by its component fields with respect to the chosen underlying even manifold \(i_C\colon M\times C\to N\).
\begin{align}
	\xi &= i_C^*X \in{\VSec{\varphi_C^*\tangent{N}}}_0, \\
	\zeta &= s^\alpha\otimes i_C^*\nabla_{F_\alpha}X \in{\VSec{\dual{S}\otimes\varphi_C^*\tangent{N}}}_0, \\
	\sigma &= -\frac12i_C^*\DLaplace X \in{\VSec{\varphi_C^*\tangent{N}}}_0.
\end{align}
This allows to equip \(\underline{W_\phi}\) with the structure of a super Fréchet domain as follows:
\begin{equation}
	\begin{split}
		\underline{W_\phi}(C)
		&= {\left(\VSec{\Phi^*\tangent{N}}\otimes\cO_C\right)}_0 \\
		&\cong \VSec{\phi^*\tangent{N}}\otimes{\left(\cO_C\right)}_0 \oplus \VSec{\dual{S}\otimes\phi^*\tangent{N}}\otimes{\left(\cO_C\right)}_1 \oplus \VSec{\phi^*\tangent{N}}{\otimes\left(\cO_C\right)}_0.
	\end{split}
\end{equation}

The exponential map on \(N\) allows to use a subsuperdomain of \(\underline{W_\phi}\) as a chart for the Fréchet supermanifold \(\underline{\Fieldspace}\) as follows:
The exponential map on the Riemannian manifold \(N\) is defined on an open neighborhood \(U\subset \tangent{N}\) of the zero section.
It yields a diffeomorphism between \(U\) and an open neighborhood \(V\) of the diagonal in \(N\times N\) via \((\pi, \exp)\), where \(\pi\colon \tangent{N}\to N\) is the projection of the vector bundle.
Any vector field \(X\in{\left(\VSec{\Phi_C^*\tangent{N}}\otimes\right)}_0\) yields a map \(\tilde{X}\colon M\to \tangent{N}\) such that \(\pi\circ\tilde{X}=\Phi\).
If the image of the reduction of \(X\) lies in \(U\), that is, \(\im \Red{X}\subset U\), we define \(\exp_\Phi X\colon M\to N\) as \(\exp\circ\tilde{X}\).
\begin{diag}\matrix[mat](m) {
		\Phi_C^*\tangent{N} & \tangent{N} & N \\
		M\times C & N & \\
	} ;
	\path[pf] {
		(m-1-1) edge (m-1-2)
			edge (m-2-1)
		(m-1-2) edge node[auto]{\(\exp\)} (m-1-3)
			edge (m-2-2)
		(m-2-1) edge[bend left=40] node[auto]{\(X\)} (m-1-1)
			edge node[auto]{\(\tilde{X}\)} (m-1-2)
			edge[bend right=10,dotted] node[auto, swap]{\(\exp_\Phi X\)} (m-1-3)
			edge node[auto,swap]{\(\Phi_C\)} (m-2-2)
	};
\end{diag}
Conversely, for any map \(\Phi'\colon M\times C\to N\) such that \(\im (\Red{\Phi}, \Red{\Phi'})\subset V\) there is a vector field \(X\in{\VSec{\Phi_C^*\tangent{N}}}_0\) such that \(\exp_\Phi X= \Phi'\).
Hence, we obtain a bijection between the open set
\begin{equation}
	\underline{U_\phi}(C) = \Set{X\in{\VSec{\Phi_C^*\tangent{N}}}_0\given \im \Red{X}\subset U}
\end{equation}
and the open Fréchet superdomain
\begin{equation}
	\underline{V_\phi}(C) = \Set{\Phi'\colon M\times C\to N\given \im (\Red{\Phi}, \Red{\Phi'})\subset V}.
\end{equation}

The component fields \(\varphi'\), \(\psi'\) and \(F'\) of \(\Phi'=\exp_\Phi X\) depend on the component fields \((\xi, \zeta, \sigma)\) of \(X\) as follows:
\begin{equation}\label{eq:ComponentFieldsExp}
	\begin{aligned}
		\varphi'
		&= \exp \circ \tilde{X}\circ i
		= \tilde{\xi}\circ \exp \\
		\psi'
		&= \left.i^*\tilde{X}^*\differential{\exp}\circ i^*\differential{\tilde{X}}\right|_\cD
			= \tilde{\xi}^*\differential{\exp}\circ\left(\zeta, 0\right) \\
		F'
		&= -\frac12i^*\DLaplace \left(\exp\circ\tilde{X}\right)
		= \tilde{\xi}^*\differential{\exp}\left(\sigma, 0\right) + f'(\varphi, \xi, \zeta)
	\end{aligned}
\end{equation}
Here \(\tilde{\xi}\colon \Red{M}\times C\to \tangent{N}\) is the map obtained from \(\xi\in\VSec{\varphi^*\tangent{N}}\) and \((\zeta, 0)=\left.i^*\differential{\tilde{X}}\right|_\cD\) is the decomposition with respect to the splitting \(\tilde{\xi}^*\tangent{\tangent{N}}=\varphi^*\tangent{N}\oplus \tangent{M}\) obtained from the connection \(\nabla\).
The dependence of \(F'\) on \(\xi\) and \(\zeta\) is not made explicit here and contained in the term \(f'(\phi, \xi, \zeta)\).
For the following it is important that  \(\varphi'\) depends only on \(\phi\) and \(\xi\), \(\psi'\) depends on \(\phi\), \(\xi\) and linearly on \(\zeta\) and \(F'\) depends in an affine linear way on \(\sigma\), as well as on \(\xi\) and \(\zeta\).

We will now construct a split subsupermanifold \(\underline{\hat{\Fieldspace}}\) such that \(\mathcal{M}(A)\subset\hat{\Fieldspace}\subset{\Fieldspace}\) which will allow to show splitness of the moduli space \(\mathcal{M}(A)\) later.
\begin{defn}\label{defn:RestrictedFieldspace}
	Let \(\underline{\hat{\Fieldspace}}\subset\underline{\Fieldspace}\) be the subfunctor of all maps \(\Phi\colon M\times C\to N\) such that its component field \(F=-\frac12i^*\DLaplace\Phi\) vanishes.
	That is
	\begin{equation}
		\begin{split}
			\underline{\hat{\Fieldspace}}\colon \cat{SPoint}^{op}&\to \cat{Sets} \\
			C &\mapsto \underline{\hat{\Fieldspace}}(C) = \Set{\Phi\in\underline{\Fieldspace}(C)\given i^*\DLaplace\Phi = 0}
		\end{split}
	\end{equation}
	We denote the inclusion by \(\underline{j_{\hat{\Fieldspace}}}\colon \underline{\hat{\Fieldspace}}\to \underline{\Fieldspace}\).
	There is also a projection \(\underline{p_{\hat{\Fieldspace}}}\colon \underline{\Fieldspace}\to \underline{\hat{\Fieldspace}}\) which sends the map \(\Phi\colon M\times C\to N\) given by component fields \((\varphi, \psi, F)\) to the map \(\hat{\Phi}\colon M\times C\to N\) whose component fields are given by \((\varphi, \psi, \hat{F}=0)\).
\end{defn}

\begin{prop}
	There is a Fréchet supermanifold structure on \(\underline{\hat{\Fieldspace}}\) such that the inclusion~\(\underline{j_{\hat{\Fieldspace}}}\) and the projection~\(\underline{p_{\hat{\Fieldspace}}}\) are supersmooth maps.
\end{prop}
\begin{proof}
	Consider the Fréchet super vector space
	\begin{equation}
		\underline{\hat{W}_\phi}(C)
		= \VSec{\phi^*\tangent{N}}\otimes{\left(\cO_C\right)}_0 \oplus \VSec{\dual{S}\otimes\phi^*\tangent{N}}\otimes{\left(\cO_C\right)}_1
	\end{equation}
	and its open subsuperdomain
	\begin{equation}
		\underline{\hat{U}_\phi}(C) = \Set{\hat{X}\in \underline{W_\phi}(C)\given \exp_\phi \Red{\hat{X}}\text{ is defined}}.
	\end{equation}
	We construct a chart by sending \(\hat{X}=(\xi, \zeta)\in\underline{\hat{U}_\phi}(C)\) to the map \(\hat{\Phi}\colon M\times C\to N\) whose component fields are given by
	\begin{equation}\label{eq:ComponentFieldsExpHat}
		\begin{aligned}
			\hat{\varphi}
			&= \tilde{\xi}\circ \exp, &
			\hat{\psi}
			&= \tilde{\xi}^*\differential{\exp}\circ\left(\zeta, 0\right), &
			\hat{F}
			&= 0,
		\end{aligned}
	\end{equation}
	compare Equations~\eqref{eq:ComponentFieldsExp}.
	This gives a set of charts \(\underline{\hat{U}_\phi}\) for \(\underline{\hat{\Fieldspace}}\) in the same way as \(\underline{U}_\phi\) gives a set of charts for \(\underline{\Fieldspace}\).

	The inclusion \(\underline{j_{\hat{\Fieldspace}}}\colon \underline{\hat{\Fieldspace}}\to \underline{\Fieldspace}\) is given in the charts \(\underline{\hat{U}_\phi}\to \underline{U_\phi}\) by \((\xi, \zeta)\to (\xi, \zeta, \sigma)\) where \(\sigma\) is the unique element of \(\VSec{\phi^*\tangent{N}}\otimes{\left(\cO_C\right)}_0\) such that \(F'=0\) in Equations~\eqref{eq:ComponentFieldsExp}.
	Such a \(\sigma\) does always exist because \(F'\) depends affine linearly on \(\sigma\).

	The projection \(\underline{p_{\hat{\Fieldspace}}}\colon \underline{\Fieldspace}\to\underline{\hat{\Fieldspace}}\) is given in the charts \(\underline{\hat{U}_\phi}\) and \(\underline{U_\phi}\) by \((\xi, \zeta, \sigma)\to (\xi, \zeta)\).
\end{proof}
Note that \(\hat{H}=\underline{\hat{\Fieldspace}}(\R^{0|0})\) is the smooth manifold of maps \(\phi\colon\Red{M}\to N\), denoted \(\mathcal{B}\) in~\cite{McDS-JHCST}.
The normal bundle \(N_{\hat{\Fieldspace}/H}\) of the embedding \(i_{\hat{H}}\colon \hat{H}\to \hat{\Fieldspace}\) has the fiber \(\VSec{\dual{S}\otimes\phi^*\tangent{N}}\) above \(\phi\).
\begin{prop}
	The charts \(\underline{\hat{W}_\phi}\) constitute a split atlas of the supermanifold \(\hat{\Fieldspace}\).
\end{prop}
\begin{proof}
	This follows because \(\varphi\) depends only on \(\xi\) and \(\hat{\psi}\) depends linearly on \(\zeta\) in Equations~\eqref{eq:ComponentFieldsExpHat}.
	Consequently, any coordinate change \(\underline{\hat{U}_\phi}\to \underline{\hat{U}_{\phi'}}\), \((\xi, \zeta)\mapsto (\xi', \zeta')\) is such that \(\xi'\) depends only on \(\xi\) and \(\zeta'\) depends linearly on \(\zeta\).
\end{proof}

\subsection{The moduli space}\label{SSec:TheModuliSpace}
In this section we recall some details on the vector bundle \(\Targetbundle\to \Fieldspace\) and its section~\(\mathcal{S}\) given by the operator \(\DJBar\).
The moduli space \(\mathcal{M}(A)\) is the zero locus of \(\mathcal{S}\).
By studying \(\differential{\mathcal{S}}\) we calculate the fibers of the deformation and obstruction bundle under the assumption of a smoothly obstructed moduli space, see Section~\ref{defn:SmoothlyObstructed}.
If in addition, the target \(N\) is Kähler and the domain \(M\) is holomorphically relatively split, we construct an explicit split atlas of the moduli space \(\mathcal{M}(A)\).

The vector bundle \(\Targetbundle\) is trivialized using parallel transport along geodesics:
\begin{equation}
	\id_{\dual{\cD}}\otimes P^{\nablabar}_{\exp_\Phi tX} \colon \dual{\cD}\otimes \Phi^*\tangent{N} \to\dual{\cD}\otimes {\left(\exp_\Phi X\right)}^*\tangent{N}.
\end{equation}
This trivialization is locally of the form \(\underline{U_\phi}\times \underline{F_\phi}\to \underline{U_\phi}\), where the fiber is equipped with the structure of a Fréchet super vector space via component field decompositions:
\begin{equation}\label{eq:ComponentsF}
	\begin{split}
		\underline{F_\phi}(C)
		={}& {\left({\VSec{\dual{\cD}\otimes\Phi^*\tangent{N}}}^{0,1}\otimes\cO_C\right)}_0 \\
		\cong{}& \VSec{\dual{S}\otimes\phi^*\tangent{N}}^{0,1}\otimes{\left(\cO_C\right)}_1 \oplus \left(\VSec{\phi^*\tangent{N}}\oplus\VSec{\cotangent{M}\otimes\phi^*\tangent{N}}^{0,1}\right)\otimes{\left(\cO_C\right)}_0 \\
		& \oplus\VSec{\dual{S}\otimes\phi^*\tangent{N}}^{0,1}\otimes{\left(\cO_C\right)}_1.
	\end{split}
\end{equation}

The section \(\underline{\mathcal{S}}\colon \underline{\Fieldspace}\to \underline{\Targetbundle}\) is given in the above trivializations by a map
\begin{equation}
	\underline{\mathcal{F}_\phi}\colon \underline{U_\phi} \to \underline{F_\phi}
\end{equation}
which is given on \(C\)-points by
\begin{equation}
	\underline{\mathcal{F}_\phi}(C) X
	= {\left(\id_{\dual{\cD}}\otimes P_{\exp_{\Phi} tX}^{\nablabar}\right)}^{-1}\DJBar \exp_{\Phi} X.
\end{equation}
The essential part of the map \(\differential{\mathcal{S}}\) is given by \(\differential{\mathcal{F}_\phi}\).
The differential of \(\underline{\mathcal{F}_\phi}\) at \(X\) in direction~\(Y\) can be determined as
\begin{equation}
	\left(\differential{\underline{\mathcal{F}}_\phi(C)} X\right)Y
	= \left.\frac{\d}{\d{s}}\right|_{s=0} {\left(\id_{\dual{\cD}}\otimes P_{\exp_{\Phi} t\left(X+sY\right)}^{\nablabar}\right)}^{-1}\DJBar \exp_{\Phi} \left(X+sY\right).
\end{equation}
We have calculated \(\differential{\mathcal{F}_\phi}\) at \(X=0\) in~\cite[Section~4.1]{KSY-SJC}.
For the result we need to express \(Y\in{\left(\VSec{\Phi^*\tangent{N}}\otimes\cO_C\right)}_0\) in the component fields
\begin{align}
	\xi &\in \VSec{\phi^*\tangent{N}}\otimes{\left(\cO_C\right)}_0, &
	\zeta &\in \VSec{\dual{S}\otimes\phi^*\tangent{N}}\otimes{\left(\cO_C\right)}_1, &
	\sigma &\in \VSec{\phi^*\tangent{N}}\otimes{\left(\cO_C\right)}_0.
\end{align}
Roughly, \(\xi\) is the variation of \(\varphi\), \(\zeta\) is the variation of \(\psi\) and \(\sigma\) is the variation of \(F\).
Then,
\begin{equation}
	\left(\differential{\underline{\mathcal{F}_\phi}}0\right)Y=
	\left(\zeta^{0,1},
	\frac14\sigma,
	-D_{\varphi}\xi,
	-\hat\Dirac_{\varphi}\zeta\right)
\end{equation}
using the component field decomposition~\eqref{eq:ComponentsF}, where
\begin{align}
	D_\varphi \xi &= \frac12\left(1+\ACI\otimes\targetACI\right)\left(\nabla\xi - \frac12\id_{\cotangent{\Smooth{M}}}\otimes\left(\targetACI\left(\nabla_\xi \targetACI\right)\right)\differential\varphi\right) \\
	\hat\Dirac_{\varphi} \zeta
	&= \Dirac^{1,0}\zeta^{1,0} - \frac12\left(\ACI \gamma^k\otimes\left(\nabla_k\targetACI\right)\zeta^{0,1} - \Tr_g\left<\left(\gamma\otimes\id_{\varphi^*\tangent{N}}\right)\zeta^{0,1}, \targetACI\varphi^*\nabla\targetACI\right>\differential{\varphi}\right).
\end{align}
The operator \(D_\varphi\) is the linearization of \(\DelJBar\) at \(\varphi\).
Both, \(\ker D_\phi\) and \(\Coker D_\phi\) have been studied in the theory of classical \(\targetACI\)-holomorphic curves, see~\cite[Chapter~7.2]{McDS-JHCST}.
If \(N\) is Kähler, the operator \(\hat{\Dirac}_\varphi\zeta\) coincides with
\begin{equation}
	\Dirac^{1,0}=\frac12\left(1+\ACI\otimes\targetACI\right)\Dirac \colon \VSec{\dual{S}\otimes_\C\varphi^*\tangent{N}}\to{\VSec{\dual{S}\otimes\varphi^*\tangent{N}}}^{0,1},
\end{equation}
see~\cite[Chapter~4.2]{KSY-SJC}.
Again if \(N\) is Kähler, \(\Coker \hat{\Dirac}_\varphi\) coincides with \(\ker\Dirac^{0,1}\), where
\begin{equation}
	\Dirac^{0,1}=\frac12\left(1+\ACI\otimes\targetACI\right)\Dirac \colon {\VSec{\dual{S}\otimes\varphi^*\tangent{N}}}^{0,1}\to\VSec{\dual{S}\otimes_\C\varphi^*\tangent{N}}.
\end{equation}
Both \(D_\varphi\) and \(\hat\Dirac_\varphi\) are \(\cO_C\)-linear and hence the fibers of deformation and obstruction bundle above \(\Phi=(\varphi=\phi\times\id_C, \psi=0, F=0)\) are given by
\begin{align}
	\underline{\DefBundle}_\Phi(C) &= {\left(\left(\ker D_\phi \oplus \ker \hat\Dirac_\phi\right)\otimes \cO_C\right)}_0 \\
	\underline{\ObsBundle}_\Phi(C) &= {\left(\left(\Coker D_\phi \oplus \Coker \hat\Dirac_\phi\right)\otimes \cO_C\right)}_0
\end{align}

Let now \(M(A)=\underline{\mathcal{M}(A)}(\R)\) be the reduced space of the moduli space of super \(\targetACI\)-holomorphic curves and \(i_{M(A)}\colon M(A)\to \mathcal{M}(A)\) the embedding.
The points of \(M(A)\) are classical \(\targetACI\)-holomorphic curves \(\phi\colon \Red{M}\to N\).
Under the assumption that the \(\mathcal{M}(A)\) is smoothly obstructed, we have \(\tangent[\phi]{M(A)}=\ker D_\phi\) and
\begin{align}
	i_{M(A)}^*\DefBundle_\phi &= \ker D_\phi\oplus \ker \hat\Dirac_\phi &
	i_{M(A)}^*\ObsBundle_\phi &= \Coker D_\phi\oplus \Coker \hat\Dirac_\phi
\end{align}
are super vector bundles with odd part \(\ker D_\phi\) and  \(\Coker \hat\Dirac_\phi\) respectively.
Consequently, the normal bundle of the embedding \(i_{M(A)}\colon M(A)\to\mathcal{M}(A)\) is given by \(N_A = N_{\mathcal{M}(A)/M(A)}\).
The fiber of \(N_A\) at the point \(\phi\) is given by holomorphic sections of \(\dual{S}\otimes_\C \phi^*\tangent{N}\), that is \(\ker \hat{\Dirac}_\phi=H^0(\dual{S}\otimes\phi^*\tangent{N})\).

\ModuliSpaceIsSplit{}
\begin{proof}
	If the gravitino vanishes and the target is Kähler, the component field equations of \(\left.\mathcal{S}\right|_{\hat{\Fieldspace}}=0\) reduce to
	\begin{align}
		0 &= \DelJBar\varphi, &
		0 &= \left(1+\ACI\otimes\targetACI\right)\psi, &
		0 &= \Dirac\psi.
	\end{align}
	The first equation does not depend on \(\psi\) at all and defines, by the smooth obstruction assumption a smooth submanifold \(j\colon M(A)\hookrightarrow \underline{\hat{\Fieldspace}}(\R^{0|0})\).
	While equations for the twisted spinor depend in a non-linear fashion on \(\varphi\), they are linear in \(\psi\).
	Those linear equations define the vector bundle \(N_{\mathcal{M}(A)/M(A)}\subset j^*N_{\hat{\Fieldspace}/\hat{H}}\).
	Hence, \(\mathcal{M}(A)\) is the split supermanifold obtained from \(M(A)\) and \(N_{\mathcal{M}(A)/M(A)}\).
\end{proof}

\subsection{Complex Structures}\label{SSec:ComplexStructures}
In this section we explain how the almost complex structure \(\targetACI\) of the target induces almost complex structures on \(\tangent{\Fieldspace}\), \(\DefBundle\) and \(\ObsBundle\).

An infinitesimal deformation of a map \(\Phi\colon M\times C\to N\) is given by an even section of \(\Phi^*\tangent{N}\).
Consequently, the fiber of the tangent bundle \(\underline{\tangent{\Fieldspace}}\) above the point \(\Phi\in\underline{\Fieldspace}(C)\) is given by
\begin{equation}
	\underline{\tangent[\Phi]{\Fieldspace}}(C) = \VSec{\Phi^*\tangent{N}}_0.
\end{equation}
In this description it is clear that \(\tangent{\Fieldspace}\) carries an almost complex structure \(\targetACI_\Fieldspace\) given by \(\Phi^*\targetACI\) induced from the almost complex structure \(\targetACI\) on \(N\).

The charts \(\exp_\Phi\colon U_\phi\to \Fieldspace\) give local trivializations of the tangent bundle \(\tangent{\Fieldspace}\) as
\begin{equation}\label{eq:TrivializationTHExponentialMap}
	\underline{\tangent{\Fieldspace}}|_{U_\phi} \simeq \underline{U_\phi}\times \underline{\VSec{\Phi^*\tangent{N}}}
\end{equation}
via the differential of the exponential map.
More explicitly, let \(X\in\underline{U_\phi}(C)\) and \(Y\in \underline{W_\phi}(C)\).
Then the constant section \(X\mapsto (X, Y)\) of \(\tangent{U_\phi}\to U_\phi\) is mapped to
\begin{equation}
	{\left(\differential{\exp_\Phi}\right)}_X(Y)
	=\left.\frac{\d}{\d\tau}\right|_{\tau=0} \exp_\Phi(X+\tau Y)
	\in \underline{\tangent[\exp_\Phi X]{\Fieldspace}}(C)
	= {\VSec{{\left(\exp_\Phi X\right)}^*\tangent{N}}}_0.
\end{equation}
The vector field \({\left(\differential{\exp_\Phi}\right)}_X(Y)\) has a description in terms of Jacobi fields along the geodesic \(\exp_\Phi tX\):
A Jacobi field \(T\) along a geodesic \(c\) is uniquely determined by \(T(t=t_0)\) and \(\dot T(t=t_0)\) at a given time and the Jacobi equation
\begin{equation}\label{eq:JacobiFields}
	\nabla_{\partial_t}\nabla_{\partial_{\partial_t}} T + R(T, \dot c)\dot c =0,
\end{equation}
see, for example,~\cite[Chapter~4.2]{JJ-RGGA}.
The same works in our situation with additional parametrization along the map \(\Phi\).
Let \(T\) be the Jacobi field along \(\exp_\Phi tX\) such that \(T(t=0)=0\) and  \(\dot{T}(t=0)=Y\).
Then \({\left(\differential{\exp_\Phi}\right)}_X(Y)=T(t=1)\).

\begin{rem}\label{rem:ChartsNotComplex}
	Unfortunately, the Jacobi equation is not compatible with the almost complex structure on \(N\).
	This can be seen in the case where \(N\) is Kähler as follows:
	\begin{equation}
		\nabla_{\partial_t}\nabla_{\partial_{\partial_t}} \targetACI T + R(\targetACI T, \dot c)\dot c,
		= \targetACI\left(\nabla_{\partial_t}\nabla_{\partial_{\partial_t}}  T + R(T, \dot c) \dot c\right) + R(\targetACI T, \dot c)\dot c - \targetACI R(T, \dot c) \dot c
	\end{equation}
	But in general \(R(\targetACI T, \dot c)\dot c\neq \targetACI R(T, \dot c)\dot c\).
	Consequently, the almost complex structure \(\targetACI_\Fieldspace\) does not coincide with \(\targetACI\) in the trivialization of \(\tangent{\Fieldspace}\) via the differential of the exponential map as in~\eqref{eq:TrivializationTHExponentialMap}.

	However, the almost complex structure \(\targetACI_\Fieldspace\) coincides with \(\targetACI\), that is \(\targetACI_{\Fieldspace}(X, Y)= (X, \targetACI Y)\) in the following trivialization not induced from charts of \(\Fieldspace\):
	Let \(\nabla\) be the Levi-Civita covariant derivative on \(N\).
	Then,
	\begin{equation}
		\nablabar X = \nabla X - \frac12 \targetACI\left(\nabla \targetACI\right)X
	\end{equation}
	is an almost complex connection, that is, \(\nablabar\targetACI=0\).
	Parallel transport along \(\exp_\Phi tX\) with respect to \(\nablabar\) is a complex linear automorphism
	\begin{equation}
		P^{\nablabar}_{\exp_\Phi tX} \colon {\VSec{\Phi^*\tangent{N}}}_0 \to {\VSec{{\left(\exp_\Phi X\right)}^*\tangent{N}}}_0
	\end{equation}
	and allows to trivialize \(\underline{\tangent{\Fieldspace}}\) in a way such that \(\targetACI_\Fieldspace(X, Y) = (X, \targetACI Y)\).
\end{rem}

\begin{lemma}
	Assume that the target \(N\) is Kähler.
	Then \(\targetACI_\Fieldspace\) restricts to an almost complex structure \(\targetACI_{\hat{\Fieldspace}}\) on \(\hat{\Fieldspace}\).
\end{lemma}
\begin{proof}
	By Definition~\ref{defn:RestrictedFieldspace}, \(\underline{\hat{\Fieldspace}}(C)\) is given by those maps \(\Phi\colon M\times C\to N\) such that \(i^*\DLaplace \Phi=0\).
	Consequently, the fibers of its tangent space are given by
	\begin{equation}
		\underline{\tangent[\Phi]{\hat{\Fieldspace}}}(C) = \Set{Y\in \underline{\tangent[\Phi]{\Fieldspace}}(C)\given i^*\DLaplace Y=0}.
	\end{equation}
	If \(N\) is Kähler, the covariant derivative \(\nabla\) and hence the operator \(\DLaplace\) are \(\targetACI\)-linear and hence if \(i^*\DLaplace Y=0\) also \(i^*\DLaplace \targetACI Y=0\).
\end{proof}
\begin{prop}\label{prop:AConDefObs}
	Assume that the moduli space \(\mathcal{M}(A)\) is smoothly obstructed and the target \(N\) is Kähler.
	Then \(\DefBundle=\tangent{\mathcal{M}(A)}\) and \(\ObsBundle\) are complex vector bundles.
\end{prop}
\begin{proof}
	The \(\UGL(1)\)-covariant differential \(\nabla^\cD\) on \(\cD\) together with the covariant derivative \(\nablabar\) on \(\tangent{N}\) induce a covariant derivative \(\nabla^{\Targetbundle}\) on \(\Targetbundle\) which satisfies
	\begin{equation}
		\nabla^{\Targetbundle}_{\overline{X}} \overline{e} = \nablabar^{\cD\otimes\Phi^*\tangent{N}}_X e
	\end{equation}
	for sections \(\overline{X}\) of \(\tangent{\Fieldspace}\) and \(\overline{e}\) of \(\Targetbundle\) which are constant in the trivializations~\ref{eq:TrivializationTHExponentialMap} of \(\tangent{\Fieldspace}\) and~\ref{eq:ComponentsF} of \(\Targetbundle\) with value \(X\in\VSec{\Phi^*\tangent{N}}\) and \(e\in\VSec{\cD\otimes\Phi^*\tangent{N}}^{1,0}\).
	The covariant derivative \(\nabla^{\Targetbundle}\) implies a decomposition \(\tangent{\Targetbundle}=\pi^*\tangent{\Fieldspace}\oplus \pi^*\Targetbundle\).
	For any map \(\Phi\colon M\times C\to N\) the differential of the section \(\mathcal{S}\colon \Fieldspace\to \Targetbundle\) is a map
	\begin{equation}
		\underline{\d_{\Phi}{\mathcal{S}}}(C)\colon \underline{\tangent[\Phi]{\Fieldspace}}(C) = \VSec{\Phi^*\tangent{N}}_0 \to \underline{\tangent[\mathcal{S}(\Phi)]{\Targetbundle}}(C) = \VSec{\Phi^*\tangent{N}}_0\oplus\VSec{\dual{\cD}\otimes \Phi^*\tangent{N}}^{0,1}_0.
	\end{equation}
	On the first summand of \(\underline{\tangent[\mathcal{S}]\Targetbundle}(C)\) the map \(\underline{\d_{\Phi}\mathcal{S}}(C)\) is the identity because \(\mathcal{S}\) is a section and on the second summand the map is
	\begin{equation}
		\left(\underline{\pi_2\circ\d_{\Phi}\mathcal{S}}(C)\right)Y
		= \frac12\left.\left(1+\ACI\otimes \targetACI\right)\left(\nabla Y - \frac12\left(\id_{\cotangent{M}}\otimes \targetACI \nabla_Y \targetACI\right) \differential{\Phi}\right)\right|_{\cD},
	\end{equation}
	see~\cite[Remark~4.13]{KSY-SJC}.
	As \(N\) is Kähler, the summand with the covariant derivative of \(\targetACI\) vanishes.

	If we equip \(\Targetbundle\) with the almost complex structure given by \(\id\otimes \targetACI\) on \(\VSec{\dual{\cD}\otimes \Phi^*\tangent{N}}^{0,1}_0\), \(\underline{\d_{\Phi}{\mathcal{S}}}(C)\) is complex linear.
	Consequently, kernel and cokernel of \(\differential{\mathcal{S}}\) carry an almost complex structure.
\end{proof}

\begin{rem}\label{rem:Integrability}
	As in the classical case, we expect the almost complex structure on \(\tangent{\mathcal{M}(A)}\) to be integrable if the target is Kähler, compare~\cite[Remark~3.2.6]{McDS-JHCST}.
	However, the split atlas of \(\mathcal{M}(A)\) from Theorem~\ref{thm:ModuliSpaceIsSplit} does not yield a holomorphic split atlas because the charts obtained via implicit function theorem are not holomorphic charts as explained in Remark~\ref{rem:ChartsNotComplex}.
\end{rem}
 
\counterwithin{equation}{section}
\section{Automorphisms of \texorpdfstring{\(\ProjectiveSpace[\C]{1|1}\)}{PC1|1}}\label{Sec:AutomorphismsOfPC11}
The moduli spaces of super stable curves and super stable maps of genus zero arise as quotients of products of copies of \(\ProjectiveSpace[\C]{1|1}\) and \(\mathcal{M}(A)\)by the superconformal automorphisms of \(\ProjectiveSpace[\C]{1|1}\), see~\cite[Section~4.3]{KSY-SQCI} as well as Propositions~\ref{prop:TorusActionM0k} and~\ref{prop:TorusActionM0kA} below.
In order to study torus actions on those quotients we need a detailed understanding of how far the superconformal automorphisms respect the split structure of \(\ProjectiveSpace[\C]{1|1}\) and \(\mathcal{M}(A)\).
After recalling basic properties of the group \(\SCP\) of superconformal automorphisms of \(\ProjectiveSpace[\C]{1|1}\) we show that the subgroup of automorphisms preserving the split structure is \(\SGL_{\C}(2)\).
Furthermore, we give detailed expressions for the action of \(\SGL_{\C}(2)\) and its left cosets \(\SUSY = \SGL_{\C}(2)\diagdown\SCP\) on \(B\)-points of \(\ProjectiveSpace[\C]{1|1}\) and \(\mathcal{M}(A)\).

Recall that a \(B\)-point \(p\colon B\to \ProjectiveSpace[\C]{1|1}\) can be described in several different ways.
The representation \([P_1:P_2:\Pi]\), \(P_i\in{\left(\cO_B\right)}_0\), \(\Pi\in{\left(\cO_B\right)}_1\) via superprojective coordinates is not unique because for any invertible \(\lambda\in \cO_B\), \([\lambda P_1:\lambda P_2:\lambda\Pi]\) represents the same point.
The representation of \(p\) via superconformal coordinates is unique but coordinate dependent:
\begin{align}
	p^\#z_1 &= p_1, &
	p^\#\theta_1 &= \pi_1; &
	p^\#z_2 &= p_2, &
	p^\#\theta_2 &= \pi_2.
\end{align}
Here,
\begin{align}
	p_1 &= \frac{P_1}{P_2}, &
	\pi_1 &= \frac{\Pi}{P_2}; &
	p_2 &= -\frac{P_2}{P_1}, &
	\pi_2 &= \frac{\Pi}{P_1}.
\end{align}
It was proven in~\cite[Lemma~4.2.2]{KSY-SQCI} that one can alternatively describe the \(B\)-point \(p\) via a pair \((\Smooth{p}, \sigma)\), where \(\Smooth{p}\in \underline{\ProjectiveSpace[\C]{1}}(B)\) is a \(B\)-point of \(\ProjectiveSpace[\C]{1}\) and \(\sigma\in\VSec{\Smooth{p}^*S}\).
The point \(\Smooth{p}\) is given in the coordinates \(z_i\) of \(\ProjectiveSpace[\C]{1}\) by \(p_i\) and \(\sigma\) is given by \(\Smooth{p}^*s_1 \pi_1=\Smooth{p}^*s_2 \pi_2\), where \(s_j=i^*D_j\) as in Example~\ref{ex:SJConP11}.

Recall that the group \(\underline{\GL_\C(2|1)}(B)\) acts on \(\underline{\ProjectiveSpace[\C]{1|1}}(B)\) by right-multiplication on the projective coordinates:
\begin{equation}
	[\tilde{Z}_1:\tilde{Z}_2:\tilde{\Theta}]
	= [Z_1:Z_2:\Theta]
	\begin{pmatrix}
		a & c & \gamma \\
		b & d & \delta \\
		\alpha & \beta & e\\
	\end{pmatrix}
\end{equation}
That is, any matrix \(L\in \underline{\GL_\C(2|1)}(B)\) induces a family of holomorphic automorphisms \(\Xi\colon B\times \ProjectiveSpace[\C]{1|1}\to \ProjectiveSpace[\C]{1|1}\).
Any family of holomorphic automorphisms of \(\ProjectiveSpace[\C]{1|1}\) arises from a matrix in \(\underline{\GL_\C(2|1)}(B)\), see~\cite{FK-PLSGSPAP11}.
Two matrices \(L\) and \(L'\) induce the same family of holomorphic automorphisms of \(\ProjectiveSpace[\C]{1|1}\) if and only if \(L'= \lambda L\) an even invertible function \(\lambda\in {\left(\cO_B\right)}_0\).

It was shown in~\cites[Chapter~2]{M-TNCG}{FK-PLSGSPAP11}, that the automorphism~\(\Xi\) is superconformal, that is preserves the distribution \(\cD\), if and only if it arises from an element of \(\underline{\SpGL_{\C}(2|1)}(B)\subset\underline{\GL_\C(2|1)}(B)\).
The matrix \(L\in \underline{\GL_\C(2|1)}(B)\) lies in \(\underline{\SpGL_{\C}(2|1)}(B)\) if its entries satisfy
\begin{equation}\label{eq:Sp21}
	\begin{aligned}
		ad - bc - \gamma\delta &= 1, &
		c\alpha -a\beta &= e\gamma, \\
		e^2 + 2\alpha\beta &= 1, &
		d\alpha -b\beta &= e\delta.
	\end{aligned}
\end{equation}
The elements of \(\underline{\SpGL_\C(2|1)}(B)\) satisfy also
\begin{align}
	\alpha\beta &= \gamma\delta, &
	e\alpha &= a\delta - b\gamma, &
	e\beta &= c\delta - d\gamma.
\end{align}

The two matrices \(L\) and \(-L\) induce the same superconformal automorphism of \(\ProjectiveSpace[\C]{1|1}\).
That is, the group of superconformal automorphisms of \(\ProjectiveSpace[\C]{1|1}\) is given by
\begin{equation}
	\SCP = \faktor{\SpGL_\C(2|1)}{\Z_2}.
\end{equation}
Note that the quotient by \(\Z_2\) identifies two connected components of \(\SpGL_\C(2|1)\) characterized by \(\Red{e}=\pm 1\).
For the remainder of this article we choose to realize \(\SCP\subset \SpGL_\C(2|1)\) as the connected component such that \(\Red{e}=1\).
Both \(\SpGL_\C(2|1)\) and \(\SCP\) are complex super Lie groups of complex dimension \(3|2\).

A matrix \(L\in \underline{\SCP}(B)\) induces the automorphisms \(\Xi\colon \ProjectiveSpace[\C]{1|1}\times B\to \ProjectiveSpace[\C]{1|1}\times B\) over \(B\) which is given in the coordinates \((z_i, \theta_i)\) by
\begin{align}
	\Xi^\#z_1
	&= \frac{a z_1 + b + \theta_1\alpha}{cz_1 + d + \theta_1\beta}
	= \frac{az_1 + b}{cz_1+ d} + \theta_1\frac{e\left(\gamma z_1 + \delta\right)}{{\left(c z_1 + d\right)}^2} \\
	\Xi^\#\theta_1
	&= \frac{\gamma z_1 + \delta + \theta_1 e}{c z_1 + d + \theta_1 \beta}
	= \frac{\gamma z_1 + \delta}{c z_1 + d} + \theta_1\frac{1-\gamma\delta}{e\left(cz_1 + d\right)} \\
	\Xi^\#z_2
	&= -\frac{c - dz_2 + \theta_2\beta}{a - bz_2 + \theta_2\alpha}
	= -\frac{c - dz_2}{a - bz_2} - \theta_2\frac{e\left(\gamma - \delta z_2\right)}{{\left(a - bz_2\right)}^2} \\
	\Xi^\#\theta_2
	&= \frac{\gamma - \delta z_2 + \theta_2 e}{a - b z_2 + \theta_2 \alpha}
	= \frac{\gamma - \delta z_2}{a - b z_2} + \theta_2\frac{1-\gamma\delta}{e\left(a - bz_2\right)}
\end{align}
For details of this calculation see~\cite[Example~9.4.3]{EK-SGSRSSCA}.

A superconformal automorphism \(\Xi\) is determined by the image of the three \(B\)-points \(0=[0:1:0]\), \(1=[1:1:0]\) and \(\infty=[1:0:0]\).
However, not all three \(B\)-points \(p_1\), \(p_2\) and \(p_3\) can arise as the image of \(0\), \(1\) and \(\infty\).
Instead, it was shown in~\cite[Chapter~2]{M-TNCG} and also in~\cite[Section~3.1]{KSY-SQCI} that for any three \(B\)-points \(p_1\), \(p_2\) and \(p_3\) there is an odd function \(\epsilon\in {\left(\cO_B\otimes \C\right)}_1\) determined up to sign and a unique superconformal automorphism \(\Xi\) such that \(\Xi(p_1)=0\), \(\Xi(p_2)=1_\epsilon=[1:1:\epsilon]\), and \(\Xi(p_3)=\infty\).
In other words, the moduli space \(\mathcal{M}_{0,3}\) of three distinct points of \(\ProjectiveSpace[\C]{1|1}\) is isomorphic to \(\faktor{\C^{0|1}}{\Z_2}\).

Using \(\SCP\subset \SpGL_\C(2|1)\), that is the sign choice \(e=1-\alpha\beta\), the inverse of an automorphism of \(\ProjectiveSpace[\C]{1|1}\) induced by the matrix \(L\) is induced by
\begin{equation}
	L^{-1} =
	\begin{pmatrix}
		d & -c & \beta \\
		-b & a & -\alpha \\
		-\delta & \gamma & 1-\alpha\beta\\
	\end{pmatrix}.
\end{equation}
Every element of \(\underline{\SCP}(B)\) can be decomposed uniquely as
\begin{equation}\label{eq:ElementOfSp21Decomposition}
	\begin{pmatrix}
		a & c & \gamma \\
		b & d & \delta \\
		\alpha & \beta & e\\
	\end{pmatrix}
	=
	\begin{pmatrix}
		a\left(1-\frac12\alpha\beta\right) & c\left(1-\frac12\alpha\beta\right) & 0 \\
		b\left(1-\frac12\alpha\beta\right) & d\left(1-\frac12\alpha\beta\right) & 0 \\
		0 & 0 & 1\\
	\end{pmatrix}
	\begin{pmatrix}
		1+\frac12\alpha\beta & 0 & -\beta \\
		0 & 1+\frac12\alpha\beta & \alpha \\
		\alpha & \beta & 1-\alpha\beta\\
	\end{pmatrix}
\end{equation}
where the determinant of the upper left block of the first matrix is one.
This is formalized in the following lemma:

\begin{lemma}\label{lemma:Sp21QuotientBySL2}
	There is an embedding \(\SGL_\C(2)\hookrightarrow \SCP\) of super Lie subgroups.
	The space of left cosets
	\begin{equation}
		\SUSY = \SGL_{\C}(2)\diagdown\SCP
	\end{equation}
	is isomorphic to \(\C^{0|2}\).
\end{lemma}
\begin{proof}
	The group homomorphism
	\begin{equation}
		\begin{split}
			\SGL_{\C}(2) &\hookrightarrow \SCP \\
			\begin{pmatrix}
				a & c \\
				b & d \\
			\end{pmatrix}
			&\mapsto
			\begin{pmatrix}
				a & c & 0 \\
				b & d & 0 \\
				0 & 0 & 1\\
			\end{pmatrix}
		\end{split}
	\end{equation}
	is an embedding, that is, \(\SGL_{\C}(2)\) is a subgroup of \(\SCP\).
	The group \(\SGL_{\C}(2)\) acts freely from the left on \(\SCP\) and \(\underline{\SGL_{\C}(2)}(\R^{0|0})\) acts properly on \(\underline{\SCP}(\R^{0|0}) = \underline{\SGL_{\C}(2)}(\R^{0|0})\).
	Consequently, by~\cite[Proposition~2.5.4]{KSY-SQCI} or the earlier references~\cites{BBRHP-QS}{A-STGM}{AHW-SO}, the quotient \(\SUSY=\SGL_{\C}(2)\diagdown\SCP\) exists.
	The quotient space \(\SUSY\) is of given by \(\C^{0|2}\) in the form of the matrices
	\begin{equation}
		\begin{pmatrix}
			1+\frac12\alpha\beta & 0 & -\beta \\
			0 & 1+\frac12\alpha\beta & \alpha \\
			\alpha & \beta & 1-\alpha\beta\\
		\end{pmatrix}
	\end{equation}
	of Equation~\eqref{eq:ElementOfSp21Decomposition}.
\end{proof}

\begin{rem}
	Notice that \(\SGL_{\C}(2)\) is not a normal subgroup of \(\SCP\).
	For example, the following conjugation of an element of \(\SGL_{\C}(2)\) with an element of supersymmetry type does not yield an element of \(\SGL_{\C}(2)\).
	\begin{equation}
		\begin{split}
			\MoveEqLeft
			\begin{pmatrix}
				1 + \frac12\sigma\tau & 0 & -\tau \\
				0 & 1 + \frac12\sigma\tau & \sigma \\
				\sigma & \tau & 1 - \sigma\tau \\
			\end{pmatrix}
			\begin{pmatrix}
				a & c & 0 \\
				b & d & 0 \\
				0 & 0 & 1 \\
			\end{pmatrix}
			\begin{pmatrix}
				1 + \frac12\sigma\tau & 0 & \tau \\
				0 & 1 + \frac12\sigma\tau & -\sigma \\
				-\sigma & -\tau & 1 - \sigma\tau \\
			\end{pmatrix}\\
			&=
			\begin{pmatrix}
				1 + \frac12\sigma\tau & 0 & -\tau \\
				0 & 1 + \frac12\sigma\tau & \sigma \\
				\sigma & \tau & 1 - \sigma\tau \\
			\end{pmatrix}
			\begin{pmatrix}
				a\left(1+\frac12\sigma\tau\right) & c\left(1+\frac12\sigma\tau\right) & a\tau - c\sigma \\
				b\left(1+\frac12\sigma\tau\right) & d\left(1+\frac12\sigma\tau\right) & b\tau - d\sigma \\
				-\sigma & -\tau & 1-\sigma\tau \\
			\end{pmatrix} \\
			&=
			\begin{pmatrix}
				a\left(1+\sigma\tau\right) - \sigma\tau & c\left(1+\sigma\tau\right) & a\tau - c\sigma - \tau \\
				b\left(1+\sigma\tau\right) & d\left(1+\sigma\tau\right) -\sigma\tau & b\tau - d\sigma + \sigma \\
				a\sigma + b\tau - \sigma & c\sigma + d\tau - \tau & \left(a + d -2\right)\sigma\tau + 1 \\
			\end{pmatrix}
		\end{split}
	\end{equation}
\end{rem}
\begin{rem}
	Note that the set of matrices of the supersymmetry type in Equation~\eqref{eq:ElementOfSp21Decomposition} are not closed under multiplication and do in general not form a subgroup as the following example shows:
	\begin{multline}\label{eq:SUSYNotGroup}
			\begin{psmallmatrix}
				1+\frac12\sigma\tau & 0 & -\tau \\
				0 & 1+\frac12\sigma\tau & \sigma \\
				\sigma & \tau & 1 - \sigma\tau \\
			\end{psmallmatrix}
			\begin{psmallmatrix}
				1+\frac12 \alpha\beta & 0 & -\beta \\
				0 & 1+\frac12\alpha\beta & \alpha \\
				\alpha & \beta & 1-\alpha\beta\\
			\end{psmallmatrix}
			= \\
			\begin{psmallmatrix}
				\left(1+\frac12\sigma\tau\right)\left(1+\frac12\alpha\beta\right) - \tau\alpha & -\tau\beta & -\left(1+\frac12\sigma\tau\right)\beta - \tau\left(1-\alpha\beta\right) \\
				\sigma\alpha & \left(1+\frac12\sigma\tau\right)\left(1+\frac12\alpha\beta\right) + \sigma\beta & \left(1+\frac12\sigma\tau\right)\alpha + \sigma\left(1-\alpha\beta\right)\\
				\sigma\left(1+\frac12\alpha\beta\right) + \left(1-\sigma\tau\right)\alpha & \tau\left(1+\frac12\alpha\beta\right) + \left(1-\sigma\tau\right)\beta & -\sigma\beta + \tau\alpha + \left(1-\sigma\tau\right)\left(1-\alpha\beta\right) \\
			\end{psmallmatrix}
	\end{multline}
\end{rem}
We will need the following partial inverse to Lemma~\ref{lemma:Sp21QuotientBySL2}:
\begin{lemma}\label{lemma:SUSYR01NormalSubgroup}
	Equip \(\underline{\SUSY}(\R^{0|1}) = {\left(\C^{0|2}\otimes \cO_{\R^{0|1}}\right)}_0 = \C^2\) with the group structure given by addition.
	Then \(\underline{\SUSY}(\R^{0|1})\hookrightarrow \underline{\SCP}(\R^{0|1})\) is a normal subgroup and the quotient is given by
	\begin{equation}
		\faktor{\underline{\SCP}(\R^{0|1})}{\underline{\SUSY}(\R^{0|1})}
		= \SGL_{\C}(2).
	\end{equation}
	Furthermore, \(\SGL_{\C}(2)\) acts on \(\underline{\SUSY}(\R^{0|1})\) from the left and
	\begin{equation}
		\underline{\SpGL(2|1)}(\R^{0|1}) = \SGL_{\C}(2)\ltimes \underline{\SUSY}(\R^{0|1}).
	\end{equation}
\end{lemma}
\begin{proof}
	The Equation~\eqref{eq:SUSYNotGroup} shows that matrices of supersymmetry type in general do not form a group except if all products of odd quantities vanish.
	As \(\R^{0|1}\) has only one odd variable, all products of odd quantities vanish and hence \(\underline{\SUSY}(\R^{0|1})\) is a subgroup of \(\underline{\SCP}(\R^{0|1})\).
	The following calculation shows that the conjugation of a matrix in \(\underline{\SUSY}(\R^{0|1})\) with a matrix in \(\SCP\) is again in \(\underline{\SUSY}(\R^{0|1})\).
	\begin{equation}
		\begin{split}
			&\begin{pmatrix}
				d & -c & \beta \\
				-b & a & -\alpha \\
				-\delta & \gamma & 1\\
			\end{pmatrix}
			\begin{pmatrix}
				1 & 0 & \tau \\
				0 & 1 & -\sigma \\
				\sigma & \tau & 1
			\end{pmatrix}
			\begin{pmatrix}
				a & c & \gamma \\
				b & d & \delta \\
				\alpha & \beta & 1\\
			\end{pmatrix} \\
			&=
			\begin{pmatrix}
				d & -c & \beta \\
				-b & a & -\alpha \\
				-\delta & \gamma & 1\\
			\end{pmatrix}
			\begin{pmatrix}
				a  & c & \gamma + \tau \\
				b  & d & \delta - \sigma \\
				\sigma a + \tau b + \alpha & \sigma c + \tau d + \beta & 1 \\
			\end{pmatrix} \\
			&=
			\begin{pmatrix}
				ad - bc & 0 & d\left(\gamma+\tau\right) - c\left(\delta-\sigma\right) + \beta\\
				0 & ad - bc & -b\left(\gamma+\tau\right) + a\left(\delta-\sigma\right) - \alpha \\
				-\delta a + \gamma b + \sigma a + \tau b + \alpha & -\delta c + \gamma d + \sigma c + \tau d + \beta & 1 \\
			\end{pmatrix} \\
			&=
			\begin{pmatrix}
				1 & 0 & c\sigma + d\tau \\
				0 & 1 & -\left(a\sigma + b\tau\right) \\
				a\sigma + b\tau & c\sigma + d\tau & 1 \\
			\end{pmatrix}
		\end{split}
	\end{equation}
	Hence, \(\underline{\SUSY}(\R^{0|1})\) is a normal subgroup.
	As every element of \(\underline{\SCP}(\R^{0|1})\) can be decomposed uniquely into a product of an element of \(\underline{\SGL_{\C}(2)}(\R^{0|1})\) and \(\underline{\SUSY}(\R^{0|1})\), it follows that \(\underline{\SpGL(2|1)}(\R^{0|1}) = \SGL_{\C}(2)\ltimes \underline{\SUSY}(\R^{0|1})\).
\end{proof}

We will now study the action of  types of automorphisms on coordinates, superpoints of \(\ProjectiveSpace[\C]{1|1}\) and \(\targetACI\)-holomorphic maps.

\begin{ex}[Lift of elements of \(\SGL_\C(2)\)]\label{ex:LiftOfMöbiusTransformations}
	Let \(l=\begin{psmallmatrix}a& c\\b&d\end{psmallmatrix}\) be an element of \(\underline{\SGL_\C(2)}\) and 
	\begin{equation}
		L=
		\begin{pmatrix}
			a & c & 0 \\
			b & d & 0 \\
			0 & 0 & 1 \\
		\end{pmatrix}
	\end{equation}
	its image in \(\SCP\).
	The matrix \(L\) acts on the superprojective coordinates as follows:
	\begin{equation}
		[Z_1:Z_2:\Theta] \cdot L
		= [aZ_1 + b Z_2:cZ_1+dZ_2:\Theta]
	\end{equation}
	and the induced automorphism \(\Xi\colon \ProjectiveSpace[\C]{1|1}\to\ProjectiveSpace[\C]{1|1}\) acts in the superconformal coordinates~\((z_i, \theta_i)\) by
	\begin{equation}\label{eq:SCCoordinateExpressionLiftOfMöbius}
		\begin{aligned}
			\Xi^\# z_1 &= \frac{a z_1 + b}{cz_1 + d} &
			\Xi^\# z_2 &= -\frac{c - b z_2}{a - bz_2} \\
			\Xi^\#\theta_1 &= \theta_1\frac1{cz_1+d} &
			\Xi^\#\theta_2 &= \theta_2\frac1{a - bz_2}
		\end{aligned}
	\end{equation}

	The automorphism \(\Xi\) induces a holomorphic automorphism \(\xi\colon \ProjectiveSpace[\C]{1}\times B\to \ProjectiveSpace[\C]{1}\times B\) over \(B\)  such that \(\Xi\circ i = i\circ \xi\), compare~\cite[Corollary~3.3.14]{EK-SGSRSSCA}.
	The map \(\xi\) is given in coordinates by
	\begin{align}
		\xi^\#z_1 &= \frac{az_1 + b}{cz_1+ d}, &
		\xi^\#z_2 &= \frac{c - dz_2}{a - bz_2}.
	\end{align}
	That is, \(\xi\) is the Möbius transformation induced by the matrix \(l\) under \(\SGL_\C(2)\twoheadrightarrow \PGL_\C(2)=\Aut(\ProjectiveSpace[\C]{1})\).
	Pulling the equalities \(\differential{\Xi}(D_1) = \frac1{cz_1+d} \Xi^*D_1\) and \(\differential{\Xi}(D_2) = \frac1{a-bz_2} \Xi^*D_2\) back along \(i\) we obtain a map \(s_\Xi\colon S\to \xi^*S\) such that
	\begin{align}
		s_\Xi(s_1) &= \frac{1}{cz_1+d}\xi^*s_1, &
		s_\Xi(s_2) &= \frac{1}{a-bz_2}\xi^*s_2.
	\end{align}
	Notice that the map \(s_\Xi\) is not invariant under replacing \(l\) by \(-l\).
	That is, in contrast to \(\xi\) the map \(s_\Xi\) depends really on \(l\in \SGL_\C(2)\) and not on its image in \(\PGL_\C(2)\).

	Let now \(p\colon B\to \ProjectiveSpace[\C]{1|1}\) be a \(B\)-point of \(\ProjectiveSpace[\C]{1|1}\) which is given in superprojective coordinates by \([P_1:P_2:\Pi]\) or in superconformal coordinates by \((p_i, \pi_i)\).
	The automorphism \(\Xi\) acts on superpoints of \(\ProjectiveSpace[\C]{1|1}\) by sending \(p\) to \(\tilde{p}=\Xi\circ p\).
	The point \(\tilde{p}\) is given in superprojective coordinates by \([aP_1 + bP_2:cP_1+dP_2:\Pi]\) or in superconformal coordinates by
	\begin{align}
		\tilde{p}_1 &= \frac{ap_1 + b}{cp_1 + d}, &
		\tilde{\pi}_1 &= \frac{\pi_1}{cp_1+d}; &
		\tilde{p}_2 &= -\frac{c-dp_2}{a-bp_2}, &
		\tilde{\pi}_2 &= \frac{\pi_2}{a-bp_2}.
	\end{align}
	Note that \(\tilde{p}\) is described by \((\Smooth{\tilde{p}}, \tilde{\sigma})\), where \(\Smooth{\tilde{p}}=\xi\circ \Smooth{p}\) and \(\tilde{\sigma}=\left(\Smooth{p}^*s_\Xi\right) (\sigma)\).

	Let now \(\Phi\colon \ProjectiveSpace[\C]{1|1}\times B\to N\) be a super \(\targetACI\)-holomorphic curve such that for local coordinates~\(X^a\) of \(N\)
	\begin{equation}\label{eq:SJCSCCoordinates}
		\Phi^\#X^a
		= \varphi^a(z_1) + \theta_1\psi_1^a(z_1)
		= \varphi^a(z_2) + \theta_2 \psi_2^a(z_2).
	\end{equation}
	The automorphism \(\Xi\) acts on super \(\targetACI\)-holomorphic curves by precomposing with the inverse, that is, sending \(\Phi\) to \(\Phi_\Xi=\Phi\circ \Xi^{-1}\).
	Using~\cite[Proposition~12.2.3]{EK-SGSRSSCA}, we obtain for the component fields \((\varphi_\Xi, \psi_\Xi)\) of \(\Phi_\Xi\)
	\begin{align}
		\varphi_\Xi &= \varphi\circ\xi^{-1}, &
		\psi_\Xi &= {\left(\xi^{-1}\right)}^*\psi \circ s_{\Xi^{-1}}.
	\end{align}
	More explicitly, one can also plug the inverse of the coordinate transformations~\eqref{eq:SCCoordinateExpressionLiftOfMöbius} into~\eqref{eq:SJCSCCoordinates} and obtain
	\begin{align}
		\varphi_\Xi^a(z_1) &= \varphi^a\left(\frac{dz_1 - c}{-bz_1+ a}\right), &
		\psi_\Xi &= \dual{s}_1\otimes \left(\frac1{-bz_1+a}\psi_1^a\left(\frac{dz_1 - c}{-bz_1+ a}\right) \varphi_\Xi^*\partial_{X^a}\right); \\
		\varphi_\Xi^a(z_2) &= \varphi^a\left(\frac{b + az_2}{d + cz_2}\right), &
		\psi_\Xi &= \dual{s}_2\otimes \left(\frac1{d + cz_2}\psi_0^a\left(\frac{b + az_2}{d + cz_0}\right) \varphi_\Xi^*\partial_{X^a}\right).
	\end{align}

	The action of \(L\) on both \(\ProjectiveSpace[\C]{1|1}\) and \(\mathcal{M}(A)\) is completely determined by \(\xi\) and \(s_\Xi\).
	That is, the action of \(L\) respects their splittings.
\end{ex}
\begin{ex}[Reflection of the odd direction]\label{ex:ReflectionOfTheOddDirection}
	An important special case of the previous Example~\ref{ex:LiftOfMöbiusTransformations} is the reflection of the odd directions given by
	\begin{equation}
		L_-=
		\begin{pmatrix}
			-1 & 0 & 0 \\
			0 & -1 & 0 \\
			0 & 0 & 1\\
		\end{pmatrix}
	\end{equation}
	which acts in projective coordinates by \([Z_1:Z_2:\Theta]\cdot L_-=[-Z_1:-Z_2:\Theta]=[Z_1:Z_2:-\Theta]\), that is, it multiplies the odd projective coordinate by \(-1\).
	The induced automorphism \(\Xi_-\colon \ProjectiveSpace[\C]{1|1}\to \ProjectiveSpace[\C]{1|1}\) is given in superconformal coordinates \((z_i, \theta_i)\) by
	\begin{align}
		\Xi_-^\# z_1 &= z_1, &
		\Xi_-^\# \theta_1 &= -\theta_1, &
		\Xi_-^\# z_2 &= z_2, &
		\Xi_-^\# \theta_2 &= -\theta_2.
	\end{align}

	The automorphism \(\Xi_-\) satisfies \(\Xi_-\circ i = i\) and \(\differential{\Xi}(D_1)=-D_1\) as well as \(\differential{\Xi}(D_2)=-D_2\).
	Hence the induces morphism \(s_{\Xi_-}\colon S\to S\) is given by \(s_{\Xi_-}(s)=-s\) for all sections \(s\) of \(S\).

	The reflection of the odd directions sends the superpoint \(p=[p_1:p_2:\pi]\) to \(\tilde{p}=[p_1:p_2:-\pi]\), or in superconformal coordinates by \((\tilde{p}_i, \pi_i)=(p_i, -\pi_i)\).
	The component fields of \(\Phi_{\Xi_-} = \Phi\circ \Xi_-^{-1}\) are given by \((\varphi, -\psi)\).

	The reflection of the odd directions is the only non-trivial element of \(\SGL_\C(2)\) such that the induced action \(\xi\colon \ProjectiveSpace[\C]{1}\to \ProjectiveSpace[\C]{1}\) is the identity.
\end{ex}

\begin{ex}[Supersymmetry transformations]\label{ex:SupersymmetryTransformations}
	We call matrices of the form
	\begin{equation}
		L=
		\begin{pmatrix}
			1+\frac12\alpha\beta & 0 & -\beta \\
			0 & 1+\frac12\alpha\beta & \alpha \\
			\alpha & \beta & 1-\alpha\beta\\
		\end{pmatrix}
	\end{equation}
	supersymmetry transformations.
	The matrix \(L\) acts on superprojective coordinates by
	\begin{equation}
		\begin{split}
			\MoveEqLeft{}
			\left[Z_1:Z_2:\Theta\right] \cdot L \\
			&= \left[Z_1\left(1+\frac12\alpha\beta\right) + \Theta\alpha: Z_2\left(1+\frac12\alpha\beta\right) + \Theta\beta : -Z_1\beta + Z_2\alpha + \Theta\left(1-\alpha\beta\right)\right]
		\end{split}
	\end{equation}
	and the induced automorphism \(\Xi\colon \ProjectiveSpace[\C]{1|1}\to \ProjectiveSpace[\C]{1|1}\) is given in superconformal coordinates by
	\begin{equation}\label{eq:SCCoordinateExpressionSUSY}
		\begin{aligned}
			\Xi^\# z_1 &= z_1 + \theta_1 \left(-\beta z_1 + \alpha\right) &
			\Xi^\# z_2 &= z_2 + \theta_2\left(\beta + \alpha z_2\right) \\
			\Xi^\#\theta_1 &= -\beta z_1 + \alpha + \theta_1\left(1-\frac12\alpha\beta\right) &
			\Xi^\#\theta_2 &= -\beta - \alpha z_2 + \theta_2\left(1-\frac12\alpha\beta\right)
		\end{aligned}
	\end{equation}

	There is an embedding \(j\colon \ProjectiveSpace[\C]{1}\to \ProjectiveSpace[\C]{1|1}\) which is the identity on topological spaces such that \(\Xi\circ i = j\).
	The embedding \(j\) is given in superconformal coordinates by
	\begin{align}
		j^\#z_1 &= z_1, &
		j^\#z_2 &= z_2, \\
		j^\#\theta_1 &= \alpha - \beta z_1, &
		j^\#\theta_2 &= -\alpha z_2 - \beta.
	\end{align}
	Pulling back the equalities
	\begin{align}
		\differential{\Xi}(D_1) &= \left(1-\frac12\alpha\beta - \theta\beta\right)\Xi^*D_1, &
		\differential{\Xi}(D_2) &= \left(1-\frac12\alpha\beta + \theta\alpha\right)\Xi^*D_2,
	\end{align}
	yields the map \(s_\Xi\colon S=i^*\cD\to S'=j^*\cD\) such that
	\begin{align}
		s_\Xi(s_1) &= \left(1-\frac12\alpha\beta\right) s'_1, &
		s_\Xi(s_2) &= \left(1-\frac12\alpha\beta\right) s'_2.
	\end{align}

	The image \(\tilde{p}\) of a superpoint \(p=[P_1:P_2:\Pi]\) under \(\Xi\) is given in superconformal coordinates by
	\begin{equation}\label{eq:SUSYPointPC11}
		\begin{aligned}
			\tilde{p}_1 &= p_1 + \pi_1\left(\alpha - \beta p_1\right), &
			\tilde{p}_2 &= p_2 + \pi_2\left(\alpha p_2 + \beta\right), \\
			\tilde{\pi}_1 &= \alpha - \beta p_1 + \pi_1\left(1-\frac12\alpha\beta\right), &
			\tilde{\pi}_2 &= -\beta - \alpha p_2 + \pi_2\left(1-\frac12\alpha\beta\right).
		\end{aligned}
	\end{equation}

	To describe the action of \(\Xi\) on a super \(\targetACI\)-holomorphic curve \(\Phi\colon \ProjectiveSpace[\C]{1|1}\to N\) first note that there is a map \(j'\colon \ProjectiveSpace[\C]{1}\to \ProjectiveSpace[\C]{1|1}\) such that \(\Xi^{-1}\circ i = j'\) given in superconformal coordinates by
	\begin{align}
		j^\#z_1 &= z_1, &
		j^\#z_2 &= z_2, \\
		j^\#\theta_1 &= -\alpha + \beta z_1, &
		j^\#\theta_2 &= \alpha z_2 + \beta.
	\end{align}
	With respect to this map \(j'\) the component fields \((\varphi_\Xi, \psi_\Xi)\) of \(\Phi_\Xi=\Phi\circ \Xi^{-1}\) are given by
	\begin{align}
		\varphi_\Xi &= \Phi\circ\Xi^{-1}\circ i = \Phi\circ j' \\
		\psi_\Xi &= i^*\left(\differential{\left(\Phi\circ\Xi^{-1}\right)}\right)|_\cD
		= i^*\left({\left(\Xi^{-1}\right)}^*\differential\Phi\circ \differential\Xi^{-1}\right)|_\cD
		= {\left(j'\right)}^*\differential{\Phi}|_\cD\circ s_{\Xi^-1}
	\end{align}
	Here are explicit coordinate expressions for the component fields:
	\begin{equation}\label{eq:SUSYComponentFields}
		\begin{aligned}
			\varphi_\Xi^a(z_1) &= \varphi^a(z_1) + \left(-\alpha + \beta z_1\right)\psi_1^a(z_1) \\
			\varphi_\Xi^a(z_2) &= \varphi^a(z_2) + \left(\alpha z_2 + \beta\right)\psi_2^a(z_2) \\
			\psi_\Xi &= \dual{\tilde{s}}_1\otimes \left(\left(1+\frac12\alpha\beta\right)\psi_1^a\left(z_1\right) + \left(-\alpha + \beta z_1\right)\partial_{z_1}\varphi^a(z_1)\right)\varphi_j^*\partial_{X^a} \\
			\psi_\Xi &= \dual{\tilde{s}}_2\otimes \left(\left(1+\frac12\alpha\beta\right)\psi_2^a\left(z_2\right) + \left(\alpha z_2 + \beta\right)\partial_{z_2}\varphi^a(z_2)\right)\varphi_j^*\partial_{X^a} \\
		\end{aligned}
	\end{equation}
	Here, \(\dual{\tilde{s}_i}\) is the dual basis to \(\tilde{s}_i={\left(j'\right)}^*D_i\).
	Note that
	\begin{equation}
		\dual{s}_i\circ s_{\Xi^{-1}} = \dual{\tilde{s}}_i\left(1+\frac12\alpha\beta\right).
	\end{equation}

	In the case that the basis \(B\) possesses only one odd parameter, that is \(B=\R^{0|1}\), the above formulas take a particularly simple form.
	Let us identify \(\underline{\SUSY}(\R^{0|1}) = \underline{\C^{0|2}}(\R^{0|1}) = {\left(H^0(S)\otimes \cO_{\R^{0|1}}\right)}_0 \) via
	\begin{equation}\label{eq:IdentificationSUSYR01H0S}
		\begin{split}
			s\colon \underline{\SUSY}(\R^{0|1}) &\to {\left(H^0(S)\otimes \cO_{\R^{0|1}}\right)}_0 \\
			(\alpha, \beta) &\mapsto s_1\left(\alpha -\beta z_1\right) = -s_2\left(\alpha z_2 + \beta\right)
		\end{split}
	\end{equation}
	The map \(s\) is compatible with the action of \(\PGL_\C(2)\) in the following way:
	Elements \(g\) of \(\SGL_\C(2)\) act on \(\underline{\SUSY}(\R^{0|1})\) via conjugation of the lift to \(\SCP\), see~\ref{ex:LiftOfMöbiusTransformations} and the proof of Lemma~\ref{lemma:SUSYR01NormalSubgroup}.
	That is, the matrix
	\begin{equation}
		g=
		\begin{pmatrix}
			a & c \\
			b & d \\
		\end{pmatrix}
	\end{equation}
	sends \((\alpha, \beta)\) to \((a\alpha + b\beta, c\alpha + d\beta)\).
	We denote the image \(s(a\alpha + b\beta, c\alpha + d\beta)\) by \(s_g\), the Möbius transformation corresponding to \(g\) by \(\xi\) and the automorphism of \(\ProjectiveSpace[\C]{1|1}\) corresponding to the lift of \(g\) to \(\SCP\) by \(\Xi\) as in Example~\ref{ex:LiftOfMöbiusTransformations}.
	Then \(\xi^*s_g = s_\Xi s\):
	\begin{equation}\label{eq:PGLActionOnH0S}
		\begin{split}
			\xi^*s_g
			&= \xi^*\left(s_1\left(a\alpha+b\beta - \left(c\alpha +d\beta\right)z_1\right)\right)
			= \xi^*s_1\left(a\alpha + b\beta - \left(c\alpha + d\beta\right)\frac{az_1 + b}{cz_1+d}\right) \\
			&= \xi^*s_1\frac{\left(a\alpha + b\beta\right)\left(cz_1 + d\right) - \left(c\alpha + d\beta\right)\left(az_1 + b\right)}{cz_1 + d} \\
			&= \xi^*s_1\frac{\left(ad - bc\right)\alpha - \left(ad - bc\right)\beta z_1}{cz_1+d}
			= \frac1{cz_1+d} \xi^*s_1\left(\alpha - \beta z_1\right)
			= s_\Xi s
		\end{split}
	\end{equation}

	The supersymmetry transformation in direction \(s=s(\alpha, \beta)\) of the point \(p\colon \R^{0|1}\to \ProjectiveSpace[\C]{1|1}\) given by the tuple \((\Smooth{p}, \sigma)\) with \(\Smooth{p}\in\underline{\ProjectiveSpace[\C]{1}}(\R^{0|1})\) and \(\sigma\in \Smooth{p}^*S\) simplifies~\eqref{eq:SUSYPointPC11} to
	\begin{equation}\label{eq:SUSYR01Points}
		\mathfrak{SUSY}^1_s(\Smooth{p}, \sigma) = (\Smooth{p}, \sigma + \Smooth{p}^*s).
	\end{equation}
	Similarly, the Equations~\eqref{eq:SUSYComponentFields} simplify to
	\begin{equation}\label{eq:SUSYR01Maps}
		\mathfrak{SUSY}^1_s(\varphi, \psi) = \left(\varphi, \psi - \left<s, \differential\varphi\right>\right).
	\end{equation}
	Here, \(\differential\varphi\in H^0(\cotangent{\ProjectiveSpace[\C]{1}}\otimes_\C \varphi^*\tangent{N})= H^0(\dual{S}\otimes_\C \dual{S}\otimes_\C \varphi^*\tangent{N})\) is the differential and \(\left<s, \differential\varphi\right>\in H^0(\dual{S}\otimes_\C \varphi^*\tangent{N})\) is the contraction of the spinor \(s\) with the first cospinor factor of \(\differential\varphi\).

	Finally we want to point out the difference between the supersymmetry transformations studied in this example and the ones more commonly studied, see, for example,~\cite[Section~12.2]{EK-SGSRSSCA}.
	Here, we have considered a holomorphic automorphism of the super Riemann surface \(\ProjectiveSpace[\C]{1|1}\) whereas more classical supersymmetries are infinitesimal transformations.
	Moreover classical supersymmetries do not respect the holomorphic structure and hence involve terms with gravitinos.
\end{ex}
 \counterwithin{equation}{subsection}

\section{Torus action on the moduli spaces of super stable maps of genus zero}\label{Sec:SplitnessStableMapsModuli}
In~\cite{KSY-SQCI} we have constructed moduli spaces of super stable maps of genus zero with Neveu--Schwarz punctures and fixed tree type as superorbifolds.
In this section we construct smooth torus actions, that is smooth \(\C^*\)-actions, on the moduli space of super stable curves with \(k\)-marked points \(\mathcal{M}_{0,k}\), as well as the moduli spaces of super stable maps with \(k\) marked points \(\mathcal{M}_{0,k}(A)\) and super stable maps with fixed tree type \(\mathcal{M}_T(\Set{A_\alpha})\).
Key requirement is that the target space \(N\) is Kähler.
As we have seen in Chapter~\ref{Sec:SplitnessOfModuliSpace}, the moduli space of super \(\targetACI\)-holomorphic curves in a Kähler manifold is split and its normal bundle possesses a complex structure.
Using that fiberwise group actions on the normal bundle extend to the split supermanifold and using the analysis of the automorphisms of \(\ProjectiveSpace[\C]{1|1}\) from Section~\ref{Sec:AutomorphismsOfPC11} we obtain a torus action on the moduli spaces \(\mathcal{M}_{0,k}\), \(\mathcal{M}_{0,k}(A)\) and \(\mathcal{M}_T(\Set{A}_\alpha)\).

In Section~\ref{SSec:TorusActionSplitSupermanifolds} we describe how a group action on the normal bundle yields a group action on the corresponding split supermanifold.
In Section~\ref{SSec:TorusActionM0k} and Section~\ref{SSec:TorusActionM0kA} we discuss how torus actions on \(\mathcal{M}_{0,k}\) and \(\mathcal{M}_{0,k}(A)\) can be obtained despite the fact that the torus action does not descend to the quotient right away.
Furthermore, we give a detailed construction of the normal bundles \(N_k\) and \(N_{k,A}\).
The special case \(\mathcal{M}_{0,1}(A)\) is treated in Section~\ref{SSec:TorusActionM01A} and moduli spaces of fixed tree type in Section~\ref{SSec:TorusActionMTA}.

\subsection{Torus action on split supermanifolds}\label{SSec:TorusActionSplitSupermanifolds}
In this section we discuss some basic properties of group actions on supermanifolds which leave the reduced manifold invariant.
Furthermore, we show that a fiberwise linear group action on a vector bundle \(E\) induces a group action on the supermanifold \(\Split E\) and use this to construct several examples.

The action of a super Lie group \(G\) on a supermanifold \(M\) is given by a smooth map \(a\colon G\times M\to M\) satisfying the usual axioms of an action, see~\cites[Chapter~5.2]{EK-SGSRSSCA}[Chapter~2.5]{KSY-SQCI} as well as~\cites{DM-SUSY}{CCF-MFS}.
The essential example we have in mind is the action~\(a\) of the complex torus \(\C^*\) with coordinate \(t\) on \(\C^{m|n}\) with coordinates \((z^a, \theta^\alpha)\) by
\begin{align}
	a^\# z^a &= z^a, &
	a^\# \theta^\alpha &= t \theta^\alpha,
\end{align}
that is a rescaling of the odd directions leaving the even directions invariant.
More formally, we say that the action \(a\) of \(G\) on \(M\) leaves the reduced manifold \(i\colon \Red{M}\to M\) invariant if for every \(g\in \underline{G}(B)\) the induced multiplication \(a_g\colon M\times B\to M\times B\) leaves the map \(i\) invariant, that is, \(a_g\circ i = i\).
In that case the normal bundle \(N_{M/\Red{M}}\) and its dual carry a \(G\)-action.
Conversely, we have the following lemma:

\begin{prop}\label{prop:TorusActionSplitness}
	Let \(E\to\Smooth{M}\) be a vector bundle over a manifold of only even dimensions.
	Any fiberwise linear group action of a group \(G\) induces an action of \(G\) on the split supermanifold \(\Split E\) obtained from \(E\) leaving \(i\colon \Smooth{M}\to \Split E\) invariant.
\end{prop}
\begin{proof}
	We use the notation from the proof of Lemma~\ref{lemma:NormalBundleFunctorOfPoints}.
	Assume that the vector bundle \(E\) is trivialized over a coordinate atlas \(\Set{V_i}\) of \(\Smooth{M}\).
	Denote the coordinates on \(V_i\) by \(y_i^a\) and the real framing of \(E\) over \(V_i\) by \(e_{i,\alpha}\).
	If \(V_i\cap V_j\) is non-empty we have coordinate change functions \(f_{ij}\) and a linear map \(F_{ij}=F_{ij}(x_i^a)\) between the framings
	\begin{align}
		x_j^a &= f_{ij}^a(x_i), &
		e_{j,\alpha} &= \tensor[_\alpha]{F}{_j_i^\beta}f(x_i^a)e_{i,\beta}.
	\end{align}
	The multiplication \(m_g\colon E\to E\) by \(g\in G\) is given in this trivialization by a matrix-valued function \(\tensor[_\alpha]{g}{_j^\beta} = \tensor[_\alpha]{g}{_j^\beta}(x_j^a)\) such that
	\begin{align}
		m_g(e_{j,\alpha}) &= \tensor[_\alpha]{g}{_j^\beta} e_{j,\beta}, &
		m_{\tilde{g}}(m_g(e_{j,\alpha})) &= \tensor[_\alpha]{g}{_j^\beta}\tensor[_\beta]{\tilde{g}}{_j^\gamma} e_{j,\gamma}. &
	\end{align}

	The supermanifold \(\Split E\) then has an atlas \(\Set{U_i}\) with local coordinates \((x_i^a, \eta_i^\alpha)\) and coordinate changes
	\begin{align}
		x_j^a &= f_{ij}^a(x_i), &
		\eta_j^\alpha &= \eta_i^\beta \tensor[_\beta]{F}{_i_j^\alpha}(x_i).
	\end{align}

	We define the \(G\)-action \(a\colon G\times M\to M\) in those local coordinates by
	\begin{align}
		a^\# x_j^a &= x_j^a, &
		a^\# \eta_j^\alpha &= \eta_i^\beta\tensor[_\beta]{g}{_j^\alpha}.
	\end{align}
	The map \(a\) defines a \(G\)-action because the maps \(m_g\) and their local expressions \(\tensor[_\alpha]{g}{_j^\beta}\) satisfy the properties of an action.
	The transformation behavior of \(\tensor[_\alpha]{g}{_j^\beta}\) under a change of frame assures that the action \(a\) is independent of the coordinates on \(\Split E\).
\end{proof}
\begin{ex}
	The group \(\Z_2\) acts on any vector bundle by sending the fiber to its negative.
	We obtain a corresponding \(\Z_2\)-action on \(\Split E\) which sends the odd coordinates to its negative.
\end{ex}
\begin{ex}[Torus action on complex vector bundles]
	Let \(E\) be a complex vector bundle over the manifold \(\Smooth{M}\).
	For any \(n\in \Z\) we can let the torus \(\C^*\) act on \(E\) via multiplication by \(t^n\).
	This leads for each \(n\in \Z\) to a smooth action of \(\C^*\) on \(\Split E\).

	If, moreover, the manifold \(\Smooth{M}\) is a complex manifold and \(E\to \Smooth{M}\) a holomorphic vector bundle the multiplication of the fibers by \(t^n\) yields a holomorphic torus action on the holomorphically split supermanifold \(\Split E\).
\end{ex}
\begin{ex}[{Super Riemann surface of genus zero~\(\ProjectiveSpace[\C]{1|1}\)}]\label{ex:TorusActionOnP11}
	Recall from Example~\ref{ex:P11isSplit} that \(\ProjectiveSpace[C]{1|1}=\Split \cO(1)\).
	We consider the holomorphic action of \(\C^*\) with coordinate \(t\) that rescales the fiber of \(\cO(1)\) by \(t\).
	The resulting torus action \(a\colon \C^* \times \ProjectiveSpace[\C]{1|1}\to \ProjectiveSpace[\C]{1|1}\) is given in superconformal coordinates by
	\begin{align}
		a^\# z_1 &= z_1, &
		a^\# \theta_1 &= t \theta_1, &
		a^\# z_2 &= z_2, &
		a^\# \theta_2 &= t \theta_2.
	\end{align}
	Alternatively, in superprojective coordinates the torus action is given by
	\begin{equation}
		[Z_1 : Z_2 : \Theta] \mapsto [Z_1 : Z_2 : t\Theta].
	\end{equation}
	Note that the torus action is not superconformal, that is, in general, the action does not preserve the distribution \(\cD\).
\end{ex}
\begin{ex}[{Moduli space of super \(\targetACI\)-holomorphic curves}]\label{ex:TorusActionOnModuliSpace}
	Assume, that the conditions of Theorem~\ref{thm:ModuliSpaceIsSplit} are satisfied and hence we have a smooth split structure \(\mathcal{M}(A)=\Split N_A\) for the moduli space of super \(\targetACI\)-holomorphic curves \(\ProjectiveSpace[\C]{1|1}\to N\).
	By Proposition~\ref{prop:AConDefObs}, the normal bundle~\(N_A\) carries an almost complex structure and can hence be equipped with a torus action which multiplies fibers with \(t\in \C^*\).
	This action is not holomorphic because the splitting \(\mathcal{M}(A)=\Split N_A\) is not holomorphic, see Remark~\ref{rem:Integrability}.

	The resulting torus action on \(\mathcal{M}(A)\) can be described as follows:
	Let \(\Phi\in \underline{\mathcal{M}(A)}(B)\) be a \(B\)-point.
If the super \(\targetACI\)-holomorphic curve \(\Phi\colon \ProjectiveSpace[\C]{1|1}\times B\to N\) has component fields \((\varphi, \psi)\) the torus action yields \((\varphi, t\psi)\), where \(t\psi\) is the complex multiplication with the complex number \(t\) on \(H^0\left(\dual{S}\otimes\varphi^*\tangent{N}\right)\otimes {\left(\cO_B\right)}_1\).
\end{ex}
\begin{rem}
	Proposition~\ref{prop:TorusActionSplitness} directly generalizes to families of supermanifolds which are relatively split.
\end{rem}

\subsection{Torus action on \texorpdfstring{\(\mathcal{M}_{0,k}\)}{M0k}}\label{SSec:TorusActionM0k}
In this section we want to construct a \(\C^*\)-action on the moduli space \(\mathcal{M}_{0,k}\) of \(k\geq 3\) points on \(\ProjectiveSpace[\C]{1|1}\).
We will see that the \(\C^*\)-action on \(\ProjectiveSpace[\C]{1|1}\) constructed in Example~\ref{ex:TorusActionOnP11} does not canonically induce an action on \(\mathcal{M}_{0,k}\).
On the other hand, we will show that the normal bundle \(N_k\) of the embedding \(M_{0,k}\to \mathcal{M}_{0,k}\) of the moduli space of \(k\) points in \(\ProjectiveSpace[\C]{1}\) into \(\mathcal{M}_{0,k}\) carries a canonical \(\C^*\)-action.

The moduli space \(\mathcal{M}_{0,k}\) for \(k\geq 3\) is defined in~\cite[Section~3.2]{KSY-SQCI} as \(\faktor{Z_k}{\SCP}\), where \(Z_k\subset {\left(\ProjectiveSpace[\C]{1|1}\right)}^k\) is the open subsupermanifold such that no two of the reduced points coincide.
Here \(\faktor{Z_k}{\SCP}\) is the superorbifold given by the Morita equivalence class of the super Lie groupoid \(\SCP\ltimes Z_k\).
We do not need many technical details of superorbifolds in this paper but those can be found in~\cite[Chapter~2]{KSY-SQCI}.
It is important to know that superorbifolds \(\mathcal{M}\) have an orbit functor
\begin{equation}
	\underline{\mathcal{M}}\colon \cat{SPoint^{op}}\to \cat{Top}
\end{equation}
sending a superpoint to the corresponding quotient topological space.
In particular for quotient superorbifolds such as \(\mathcal{M}_{0,k}\), we have
\begin{equation}
	\underline{\mathcal{M}_{0,k}}(B) = \faktor{\underline{Z_k}(B)}{\underline{\SCP}(B)}
\end{equation}
where
\begin{align}
	\underline{Z_k}\colon \cat{SPoint^{op}}&\to \cat{Man} &
	\underline{\SCP}\colon \cat{SPoint^{op}}\to \cat{Man}
\end{align}
are the point functors in the sense of Molotkov--Sachse of \(Z_k\) and \(\SCP\) respectively.
Both are functors from the category of superpoints to the category of manifolds.

The supermanifold \(Z_k\subset{\left(\ProjectiveSpace[\C]{1|1}\right)}^k\) is holomorphically split, and hence by Proposition~\ref{prop:TorusActionSplitness} obtains a \(\C^*\)-action \(a_{Z_k}\colon \C^*\times Z_k\to Z_k\) from rescaling the fibers of the normal bundle by \(\C^*\).
Unfortunately, the action \(a_k\) on \(Z_k\) does not descend to a \(\C^*\)-action \(a_{\mathcal{M}_{0,k}}\) on \(\mathcal{M}_{0,k}\):
\begin{diag}
	\matrix[mat](m){
		\C^*\times Z_k & Z_k \\
		\C^*\times \mathcal{M}_{0,k} & \mathcal{M}_{0,k}=\faktor{Z_k}{\SCP} \\
	} ;
	\path[pf]{
		(m-1-1) edge node[auto]{\(a_{Z_k}\)} (m-1-2)
			edge  (m-2-1)
		(m-1-2) edge (m-2-2)
		(m-2-1) edge[dotted] node[auto]{\(a_{\mathcal{M}_{0,k}}\)}(m-2-2)
	};
\end{diag}
The action \(a_{\mathcal{M}_{0,k}}\) cannot be constructed in this way because in general for \(z\in \underline{Z_k}(B)\) and \(g\in \underline{\SCP}(B)\) there is no \(\tilde{g}\in\underline{\SCP}(B)\) such that
\begin{equation}
	\underline{a_{Z_k}}(z\cdot g) = \underline{a_{Z_k}}(z)\cdot \tilde{g}
\end{equation}
as can be seen in the following counterexample:
\begin{ex}
	Consider the three points \(0=[0:1:0]\), \(1_\epsilon=[1:1:\epsilon]\) and \(\infty=[1:0:0]\) of \(\ProjectiveSpace[\C]{1|1}\) and the automorphism \(\Xi\colon \ProjectiveSpace[\C]{1|1}\to \ProjectiveSpace[\C]{1|1}\) induced from the following matrix:
	\begin{equation}
		\begin{pmatrix}
			1  & 0 & -\beta \\
			0 & 1 & 0 \\
			0 & \beta & 1
		\end{pmatrix}
	\end{equation}
	The image is given by the following matrix product:
	\begin{equation}
		\begin{bmatrix}
			0 & 1 & 0 \\
			1 & 1 & \epsilon \\
			1 & 0 & 0 \\
		\end{bmatrix}
		\begin{pmatrix}
			1  & 0 & -\beta \\
			0 & 1 & 0 \\
			0 & \beta & 1
		\end{pmatrix}
		=
		\begin{bmatrix}
			0 & 1 & 0 \\
			1 & 1 + \beta\epsilon & -\beta + \epsilon \\
			1 & 0 & -\beta \\
		\end{bmatrix}
	\end{equation}
	The torus action on the resulting three points yields
	\begin{equation}
		\begin{bmatrix}
			0 & 1 & 0 \\
			1 & 1 + \beta\epsilon & t\left(\epsilon-\beta\right) \\
			1 & 0 & -t\beta \\
		\end{bmatrix}.
	\end{equation}
	But there is no automorphism sending \(0\), \(1_{t\epsilon}\) and \(\infty\) to those three points.
\end{ex}

In order to construct a torus action on \(\mathcal{M}_{0,k}\) we give an equivalent description of \(\mathcal{M}_{0,k}\) as \(\faktor{\C^{0|1}\times\tilde{Z}_{k-3}}{\Z_2}\) for a \(\C^*\)-invariant \(\C^{0|1}\times\tilde{Z}_{k-3}\subset Z_k\) and construct the action \(a_{\mathcal{M}_{0,k}}\) by the requirement that the following diagram commutes:
\begin{diag}
	\matrix[mat](m){
		\C^*\times \left(\C^{0|1}\times\tilde{Z}_{k-3}\right) & \C^*\times Z_k & \C^{0|1}\times\tilde{Z}_{k-3} \\
		\C^*\times \mathcal{M}_{0,k} && \mathcal{M}_{0,k}=\faktor{\C^{0|1}\times\tilde{Z}_{k-3}}{\Z_2} \\
	} ;
	\path[pf]{
		(m-1-1) edge[inj] (m-1-2)
			edge  (m-2-1)
		(m-1-2) edge node[auto]{\(a_{Z_k}\)} (m-1-3)
		(m-1-3) edge (m-2-3)
		(m-2-1) edge node[auto]{\(a_{\mathcal{M}_{0,k}}\)}(m-2-3)
	};
\end{diag}
More precisely, in order to avoid the concept of morphism of superorbifolds, we will just prove that the action \(a_{Z_k}\) leaves \(\tilde{Z}_{k-3}\) invariant and commutes with the \(\Z_2\)-action.

Let \(\tilde{Z}_{k-3}\subset Z_{k-3}\) be the open subset such that none of the reduced points coincides with \(0\), \(1\) or \(\infty\).
We define a \(\Z_2\)-action on \(\C^{0|1}\times \tilde{Z}_{k-3}\) by
\begin{equation}
	(\epsilon, (z_1, \dotsc, z_{k-3})) \mapsto (-\epsilon, (\Xi_- z_1, \dotsc, \Xi_- z_{k-3})),
\end{equation}
and an embedding
\begin{equation}\label{eq:EquivariantMapZk-3}
	\begin{split}
		f\colon \C^{0|1}\times \tilde{Z}_{k-3} &\to Z_k \\
		(\epsilon, (z_1, \dotsc, z_{k-3})) &\mapsto (0, 1_\epsilon, \infty, z_1, \dotsc, z_{k-3})
	\end{split}
\end{equation}
For \(\gamma\colon \Z_2\to \SCP\) given by \(\gamma(1)=\Xi_-\), the map \(f\) is \(\gamma\)-equivariant and hence induces a groupoid homomorphism from \(F\colon\Z_2\ltimes \left(\C^{0|1}\times \tilde{Z}_{k-3}\right)\to \SCP\ltimes Z_k\).
By the following lemma, \(F\) is a Morita equivalence:
\begin{lemma}\label{lemma:MoritaEquivalence}
	Let \(a_M\colon G\times M\to M\) and \(a_{M'}\colon G'\times M'\to M\) be actions of the super Lie groups \(G\), \(G'\) on \(M\) and \(M'\) respectively.
	Assume that \(\gamma\colon G'\to G\) is a homomorphism of super Lie groups and \(f\colon M'\to M\) is \(\gamma\)-equivariant, that is, \(a_M\circ\left(\gamma\times f\right)= f\circ a_{M'}\).
	Then \(f\) induces a homomorphism of the super Lie groupoids \(F\colon G'\ltimes M'\to G\ltimes M\).
	\(F\) is a Morita equivalence if and only if
	\begin{enumerate}
		\item\label{item:MoritaEquivalence:OrbitSurjectivity}
			The composition
			\begin{diag}
				\matrix[mat](m) {
					G\times M' & G\times M & M \\
				} ;
				\path[pf] {
					(m-1-1) edge node[auto]{\(\id_G\times f\)} (m-1-2)
					(m-1-2) edge node[auto]{\(a_M\)} (m-1-3)
				};
			\end{diag}
			is a surjective submersion.
		\item\label{item:MoritaEquivalence:Restriction}
			The commutative square
			\begin{diag}
				\matrix[mat](m){
					G'\times M' & G\times M\\
					M'\times M' & M\times M\\
				} ;
				\path[pf]{
					(m-1-1) edge node[auto]{\(\gamma\times f\)} (m-1-2)
						edge node[auto]{\((p_{M'},a_{M'})\)} (m-2-1)
					(m-1-2) edge node[auto]{\((p_M,a_M)\)}(m-2-2)
					(m-2-1) edge node[auto]{\(f\times f\)}(m-2-2)
				};
			\end{diag}
			is a fiber product square, that is, the induced map
			\begin{equation}
				G'\times M' \to \left(M'\times M'\right)\times_{f\times f, (p_M, a_M)} \left(G\times M\right)
			\end{equation}
			is a superdiffeomorphism.
	\end{enumerate}
	In particular, it follows that if \(G\ltimes M\) or \(G'\ltimes M'\) is a superorbifold so is the other and
	\begin{equation}
		\faktor{M}{G}=\faktor{M'}{G'}.
	\end{equation}
\end{lemma}
The two conditions in the Lemma can be stated colloquially as follows:
The condition~\ref{item:MoritaEquivalence:OrbitSurjectivity} assures that every orbit of \(a_M\) is reached by \(a_{M'}\), whereas the condition~\ref{item:MoritaEquivalence:Restriction} implies that \(a_M\) and \(a_M'\) coincide on the image of \(f\).
The two conditions are fulfilled in particular if \(M'\) is a subsupermanifold going through all orbits of \(G\) and \(G'\) is the restriction of \(G\) to the super Lie subgroup that preserves \(M'\).

Applying Lemma~\ref{lemma:MoritaEquivalence} to the map \(f\colon \C^{0|1}\times \tilde{Z}_{k-3}\to Z_k\), see Equation~\eqref{eq:EquivariantMapZk-3}, we obtain
\begin{prop}\label{prop:TorusActionM0k}
	The moduli space of super Riemann surfaces of genus zero with \(k\) marked points is isomorphic to
	\begin{equation}\label{eq:M0k}
		\mathcal{M}_{0,k}
		= \faktor{Z_k}{\SCP}
		= \faktor{\C^{0|1}\times\tilde{Z}_{k-3}}{\Z_2}.
	\end{equation}
	The torus action on \(\C^{0|1}\times \tilde{Z}_{k-3}\) commutes with \(\Z_2\) and hence induces a torus action \(\underline{a_{\mathcal{M}_{0,k}}}\colon \underline{\C^*}\times \underline{\mathcal{M}_{0,k}}\to \underline{\mathcal{M}_{0,k}}\).
\end{prop}
Notice that the only fixed points of the torus action on the moduli space, in the sense of functor of points, are the \(\R^{0|0}\)-points \(\underline{\mathcal{M}_{0,k}}(\R^{0|0})\).
The space \(\underline{\mathcal{M}_{0,k}}(\R^{0|0})\) coincides as a manifold with \(M_{0,k}\), the Knudsen moduli space of \(k\geq 3\) marked points on \(\ProjectiveSpace[\C]{1}\).
The Proposition~\ref{prop:TorusActionM0k} gives explicit identifications
\begin{equation}
	M_{0,k}
	=\faktor{\underline{Z_k}(\R^{0|0})}{\underline{\SCP}(\R^{0|0})}
	=\faktor{\Red{\left(Z_k\right)}}{\SGL_\C(2)}
	=\Red{\left(\tilde{Z}_{k-3}\right)}
	=\underline{\tilde{Z}_{k-3}}(\R^{0|0}).
\end{equation}
Furthermore, \(M_{0,k}=\Red{\left(\tilde{Z}_{k-3}\right)}\) carries the complex normal bundle
\begin{equation}
	N_{k}=N_{\C^{0|1}\times \tilde{Z}_{k-3}/\Red{\left(\tilde{Z}_{k-3}\right)}}
\end{equation}
together with a fiberwise \(\Z_2\)-action and a fiberwise \(\C^*\)-action \(a_{N_{k}}\colon \C^*\times N_{k}\to N_{k}\) commuting with the \(\Z_2\)-action such that
\begin{itemize}
	\item
		\(\underline{\mathcal{M}_{0,k}}(\R^{0|1})=\faktor{N_k}{\Z_2}\) and the projection \(\faktor{N_k}{\Z_2}\to M_{0,k}\) coincides with the map \(\underline{\mathcal{M}_{0,k}}(\R^{0|0}\to \R^{0|1})\) and
	\item
		the \(\C^*\)-action \(a_{N_k}\) coincides with \(\underline{a_{\mathcal{M}_{0,k}}}(\R^{0|1})\) up to sign.
\end{itemize}
However, the construction of \(N_{k}\) and \(a_{N_{k}}\) above uses the choice of identification \(\mathcal{M}_{0,k}=\faktor{\C^{0|1}\times\tilde{Z}_{k-3}}{\Z_2}\).
We will now give an alternative construction of \(N_{k}\) and \(a_{N_{k}}\) which is canonical.

Consider the following short exact sequence of vector bundles over \(\Red{\left(Z_k\right)}\):
\begin{diag}\label{seq:NormalbundleZ_k}
	\matrix[mat](m) {
		0 & \Red{\left(Z_k\right)}\times {\left(H^0(S)\otimes \cO_{\R^{0|1}}\right)}_0 & N_{Z_k/\Red{\left(Z_k\right)}} & Q & 0 \\
	} ;
	\path[pf] {
		(m-1-1) edge (m-1-2)
		(m-1-2) edge node[auto]{\(\mathfrak{SUSY}^1\)} (m-1-3)
		(m-1-3) edge (m-1-4)
		(m-1-4) edge (m-1-5)
	};
\end{diag}
The vector bundle \(N_{Z_k/\Red{\left(Z_k\right)}}\) is the restriction of \(k\) copies of \(S=\cO(1)\to \ProjectiveSpace[\C]{1}\) to \(\Red{\left(Z_k\right)}={\left(\ProjectiveSpace[\C]{1}\right)}^k\).
More precisely, let \(\Smooth{p_i}\colon \Red{\left(Z_k\right)}\to \ProjectiveSpace[\C]{1}\) be the projection on the \(i\)-th factor.
Then
\begin{equation}
	N_{Z_k/\Red{\left(Z_k\right)}}
	= \bigoplus_{i=1}^k\Smooth{p_i}^*S,
\end{equation}
and the maps \(\mathfrak{SUSY}^1\) sends constant sections \(s\) to \(\bigoplus \Smooth{p}_i^*s\), compare Equation~\eqref{eq:SUSYR01Points}.

We equip the vector bundles of the short exact sequence~\eqref{seq:NormalbundleZ_k} with the \(\C^*\)-action which rescales fiberwisely by \(t\in \C^*\).
The linear maps of the short exact sequence~\eqref{seq:NormalbundleZ_k} are equivariant.

The group \(\SGL_\C(2)\) acts on the short exact sequence~\eqref{seq:NormalbundleZ_k} in an equivariant way:
For an element \(g\in\SGL_\C(2)\) we denote the corresponding Möbius transformation by \(\xi\) and the automorphism of \(\ProjectiveSpace[\C]{1|1}\) by \(\Xi\).
As explained in Equation~\eqref{eq:PGLActionOnH0S}, the group element \(g\) acts on \(\Red{\left(Z_k\right)}\times {\left(H^0(S)\otimes \cO_{\R^{0|1}}\right)}_0\) via
\begin{equation}
	(\Smooth{p_i}, s) \mapsto (\xi\circ\Smooth{p_i}, s_\Xi s).
\end{equation}
The action of \(\SGL_\C(2)\) on \(N_{Z_k/\Red{\left(Z_k\right)}}=\bigoplus\Smooth{p_i}^*S\) is given by \(\Smooth{p_i}^*s_\Xi\) on every summand.
Hence the action of \(\mathfrak{SUSY}^1\) is \(\SGL_\C(2)\) equivariant.
Moreover, the action of \(\SGL_\C(2)\) on the short exact sequence~\eqref{seq:NormalbundleZ_k} commutes with the torus action.

Consequently, we obtain a sequence of bundle maps
\begin{diag}\label{SES:Nk}
	\matrix[mat, row sep=huge](m) {
		\faktor{\Red{\left(Z_k\right)}\times {\left(H^0(S)\otimes \cO_{\R^{0|1}}\right)}_0}{\SGL_\C(2)} & \faktor{N_{Z_k/\Red{\left(Z_k\right)}}}{\SGL_{\C}(2)} & \faktor{Q}{\SGL_\C(2)}  \\
	} ;
	\path[pf] {
		(m-1-1) edge node[auto]{\(\widetilde{\mathfrak{SUSY}^1}\)} (m-1-2)
		(m-1-2) edge (m-1-3)
	};
\end{diag}
over \(M_{0,k}=\faktor{\Red{\left(Z_k\right)}}{\SGL_\C(2)}\) together with a \(\C^*\)-action on them.
The fibers of the bundles are quotients of vector bundles by the \(\Z_2\)-action sending a vector to its negative.
The map \(\widetilde{\mathfrak{SUSY}^1}\) is injective and the map to \(\faktor{Q}{\SGL_{\C}(2)}\) is surjective; both are linear up to sign.

The bundle~\(\tilde{Q}=\faktor{Q}{\SGL_\C(2)}\) is \(\C^*\)-equivariantly isomorphic to the bundle~\(\faktor{N_{k}}{\Z_2}\) obtained earlier because the vector bundle \(Q\) is isomorphic as a manifold to \(\faktor{\underline{Z_k}(\R^{0|1})}{\underline{\SUSY}(\R^{0|1})}\) and
\begin{equation}
	\begin{split}
		\tilde{Q}
		&= \faktor{\left(\faktor{\underline{Z_k}(\R^{0|1})}{\underline{\SUSY}(\R^{0|1})}\right)}{\SGL_\C(2)}
		= \faktor{\underline{Z_k}(\R^{0|1})}{\underline{\SCP}(\R^{0|1})}
		= \underline{\mathcal{M}_{0,k}}(\R^{0|1}).
	\end{split}
\end{equation}

\subsection{Torus action on \texorpdfstring{\(\mathcal{M}_{0,k}(A)\)}{M0k(A)}}\label{SSec:TorusActionM0kA}
In this section we want to construct a torus action on the moduli space of super stable maps with \(k\) marked points
\begin{equation}
	\mathcal{M}_{0,k}(A) = \faktor{Z_k\times \mathcal{M}(A)}{\SCP}
\end{equation}
for \(k>2\).
As in Section~\ref{SSec:TorusActionM0k} for \(\mathcal{M}_{0,k}\), a torus action on \(\mathcal{M}_{0,k}(A)\) cannot be directly obtained from the torus actions on \(\ProjectiveSpace[\C]{1|1}\) and \(\mathcal{M}(A)\) discussed in Example~\ref{ex:TorusActionOnP11} and~\ref{ex:TorusActionOnModuliSpace} because this torus action does not descend to the quotient.
Instead, like in Proposition~\ref{prop:TorusActionM0k}, we will give a description of \(\mathcal{M}_{0,k}(A)\) as a quotient of a smaller supermanifold by a smaller super Lie group wich allows for a torus action descending to the quotient.
In a second step we will then show that the resulting torus action on the normal bundle~\(N_{k,A}\) is canonical.

\begin{prop}\label{prop:TorusActionM0kA}
	Let \(Z'_{k-2}\) be the open subsupermanifold of \({\left(\ProjectiveSpace[\C]{1|1}\right)}^{k-2}\) such that no two reduced points coincide and none of them coincides with \(0\) or \(\infty\).
	Denote by \(G'\subset\SCP\) the super Lie group of superconformal automorphisms of \(\ProjectiveSpace[\C]{1|1}\) that preserve the points \(0\) and \(\infty\).
	Then
	\begin{equation}
		\mathcal{M}_{0,k}(A)
		= \faktor{Z_k\times \mathcal{M}(A)}{\SCP}
		= \faktor{Z'_{k-2}\times \mathcal{M}(A)}{G'}.
	\end{equation}
	Furthermore, \(Z'_{k-2}\times \mathcal{M}(A)\subset Z_k\times \mathcal{M}(A)\) is a \(\C^*\)-invariant subsupermanifold and the \(\C^*\)-action commutes with the action of \(G'\).
	Hence the torus action on \(Z'_{k-2}\times \mathcal{M}(A)\) induces a \(\C^*\)-action \(\underline{a_{\mathcal{M}_{0,k}(A)}}\colon \underline{\C^*}\times \underline{\mathcal{M}_{0,k}(A)}\to \underline{\mathcal{M}_{0,k}(A)}\) on the orbit functor.
\end{prop}
\begin{proof}
	Let
	\begin{equation}
		\begin{split}
			f\colon Z'_{k-1}\times\mathcal{M}(A) &\to Z_k\times \mathcal{M}(A)\\
			(z_1, \dotsc, z_{k-2}, \Phi) &\mapsto (0, \infty, z_1, \dotsc, z_{k-2}, \Phi)
		\end{split}
	\end{equation}
	and \(\gamma\colon G'\to \SCP\) the inclusion.
	Then \(f\) is \(\gamma\)-equivariant and we use Lemma~\ref{lemma:MoritaEquivalence} to obtain the claimed equality because any two \(B\)-points \((p_1, p_2)\) of \(\ProjectiveSpace[\C]{1|1}\) can be mapped to \((0,\infty)\) by an element of \(\SCP\).

	To see that the action of \(G'\) commutes with the torus action on \(Z'_{k-2}\times\mathcal{M}(A)\) notice that \(G'\) is given by
	\begin{equation}
		\underline{G}(B)=
		\Set{
			\begin{pmatrix}
				a & 0 & 0 \\
				0 & a^{-1} & 0\\
				0 & 0 & 1 \\
			\end{pmatrix}
			\in \underline{\SCP}(B)
		}.
	\end{equation}
	All elements of \(G'\) are lifts of elements of \(\SGL_\C(2)\) to \(\SCP\) and hence the action of \(G'\) on \(Z'_{k-2}\times \mathcal{M}(A)\) is described in Example~\ref{ex:LiftOfMöbiusTransformations} and Example~\ref{ex:ReflectionOfTheOddDirection}.
	With the formulas there one sees that the \(\C^*\)-action commutes with the action of \(G'\).
\end{proof}

The moduli space \(\mathcal{M}_{0,k}(A)\) is a superorbifold where precisely the \(\R^{0|0}\)-points are \(\Z_2\)-isotropy points.
The manifold of \(\R^{0|0}\)-points \(\underline{\mathcal{M}_{0,k}(A)}(\R^{0|0})=M_{0,k}(A)\) coincides with the moduli space of stable maps of genus zero and \(k\) marked points.

\begin{rem}\label{rem:TorusActionM0k3A}
	If \(k\geq 3\) one can show with the same methods that
	\begin{equation}
		\mathcal{M}_{0,k}(A) = \faktor{\C^{0|1}\times \tilde{Z}_{k-3} \times \mathcal{M}(A)}{\Z_2}.
	\end{equation}
	This leads to the same \(\C^*\)-action on \(\mathcal{M}_{0,k}(A)\).

	In this case we denote by a \(N_{k,A}\) the normal bundle of \(\Red{\left(\C^{0|1}\times \tilde{Z}_{k-3} \times \mathcal{M}(A)\right)}\) in \(\C^{0|1}\times \tilde{Z}_{k-3} \times \mathcal{M}(A)\) and its \(\C^*\)-action by \(a_{N_{k,A}}\).
	They satisfy
	\begin{itemize}
		\item
			\(\faktor{N_{k,A}}{\Z_2}\) is isomorphic to \(\underline{\mathcal{M}_{0,k}(A)}(\R^{0|1})\) and the projection \(\faktor{N_{k,A}}{\Z_2}\to M_{0,k}(A)\) coincides with \(\underline{\mathcal{M}_{0,k}(A)}(\R^{0|0}\to\R^{0|1})\).
		\item
			The \(\C^*\)-action \(\underline{a_{\mathcal{M}_{0,k}(A)}}(\R^{0|1})\) coincides with \(a_{N_{k,A}}\) up to a sign.
	\end{itemize}
	However, the construction of the action \(a_{N_{k,A}}\) depends on the explicit identification in Proposition~\ref{prop:TorusActionM0kA}.
	We will show that the \(\C^*\)-action on \(Z_k\times \mathcal{M}(A)\) induces a canonical \(\C^*\)-action on \(\faktor{N_{k,A}}{\Z_2}\) via a different description of \(\underline{\mathcal{M}_{0,k}(A)}(\R^{0|1})\).
\end{rem}

Let \(\tilde{M}_{0,k}(A)=\Red{Z_k}\times M(A)\), that is, \(M_{0,k}(A)=\faktor{\tilde{M}_{0,k}(A)}{\PGL_\C(2)}\).
We will write points of \(\tilde{M}_{0,k}(A)\) as \((\Smooth{p_i}, \phi)\).
Furthermore, denote by \(\widetilde{H^0(S)}\) the trivial bundle \(\tilde{M}_{0,k}(A)\times {\left(H^0(S)\otimes\R^{0|1}\right)}_0\) and by
\begin{equation}
	\begin{split}
		\mathfrak{SUSY}^1\colon \widetilde{H^0(S)} &\to N_{Z_k/\Red{\left(Z_k\right)}}\oplus N_{\mathcal{M}(A)/M(A)} \\
		s(\Smooth{p_i}, \phi) &\mapsto \bigoplus \Smooth{p_i}^*s\oplus -\left<s, \differential\phi\right>
	\end{split}
\end{equation}
compare~\eqref{eq:SUSYR01Points} and~\eqref{eq:SUSYR01Maps}.
Then by construction
\begin{equation}
	\Coker\mathfrak{SUSY}^1 = \faktor{\underline{Z_k\times \mathcal{M}(A)}(\R^{0|1})}{\underline{\SUSY}(\R^{0|1})}.
\end{equation}
We equip both \(\widetilde{H^0(S)}\) and \(N_{Z_k/\Red{\left(Z_k\right)}}\oplus N_{\mathcal{M}(A)/M(A)}\) with the \(\C^*\)-action that rescales the fibers by \(t\in \C^*\).
As \(\mathfrak{SUSY}^1\) is linear, it is also \(\C^*\)-equivariant and hence \(\Coker\mathfrak{SUSY}^1\) also carries a \(\C^*\)-action.

The group \(\SGL_\C(2)\) acts on \(\widetilde{H^0(S)}\) via \(s_\Xi\) and on \(N_{Z_k/\Red{\left(Z_k\right)}}\oplus N_{\mathcal{M}(A)/M(A)}\) as discussed in Example~\ref{ex:LiftOfMöbiusTransformations} and is equivariant as discussed in Example~\ref{ex:SupersymmetryTransformations}.
Consequently,  we obtain a sequence of bundle maps over \(M_{0,k}(A)\):
\begin{diag}\matrix[mat](m) {
		\faktor{\widetilde{H^0(S)}}{\SGL_\C(2)} &\faktor{N_{Z_k/\Red{\left(Z_k\right)}}\oplus N_{\mathcal{M}(A)/M(A)}}{\SGL_\C(2)} &\faktor{\Coker\mathfrak{SUSY}^1}{\SGL_\C(2)}\\
	} ;
	\path[pf] {
		(m-1-1) edge node[auto]{\(\mathfrak{SUSY}^1\)} (m-1-2)
		(m-1-2) edge (m-1-3)
	};
\end{diag}
As in~\eqref{SES:Nk}, the fibers of the bundles are vector bundles up to the \(\Z_2\)-action and the sequence is a short exact sequence up to sign.

Here the last bundle is isomorphic to \(\faktor{N_{k,A}}{\Z_2}\) as a bundle with \(\C^*\)-action because:
\begin{equation}
	\begin{split}
		\faktor{\Coker\mathfrak{SUSY}^1}{\SGL_\C(2)}
		&= \faktor{\left(\faktor{\underline{Z_k\times \mathcal{M}(A)}(\R^{0|1})}{\underline{\SUSY}(\R^{0|1})}\right)}{\SGL_\C(2)} \\
		&= \faktor{\underline{Z_k\times \mathcal{M}(A)}(\R^{0|1})}{\underline{\SCP}(\R^{0|1})} \\
		&= \underline{\mathcal{M}_{0,k}(A)}(\R^{0|1}).
	\end{split}
\end{equation}
Notice that this description of \(\underline{\mathcal{M}_{0,k}(A)}(\R^{0|1})\) also works in the case of \(k=1,2\) in contrast to Proposition~\ref{prop:TorusActionM0kA} and the resulting construction as \(\faktor{N_{k,A}}{\Z_2}\) in Remark~\ref{rem:TorusActionM0k3A}.

\begin{rem}
	Note that the torus action on \(a_{\mathcal{M}_{0,k}(A)}\colon \C^*\times \mathcal{M}_{0,k}(A)\to \mathcal{M}_{0,k}(A)\) constructed in Proposition~\ref{prop:TorusActionM0kA} yields a smooth orbifold morphism but not a holomorphic orbifold morphism because the construction relies on the non-holomorphic splitting \(\mathcal{M}(A)=\Split N_A\).
	In contrast, the torus action \(a_{\mathcal{M}_{0,k}}\colon \C^*\times\mathcal{M}_{0,k}\to \mathcal{M}_{0,k}\) constructed in Proposition~\ref{prop:TorusActionM0k} yields a holomorphic orbifold morphism.

	Similarly, the normal bundle \(N_k\) is a holomorphic bundle but it depends on the integrability of the almost complex structure on \(\mathcal{M}(A)\) whether or not \(N_{k,A}\) is a holomorphic bundle, compare Remark~\ref{rem:Integrability}.
\end{rem}

\begin{rem}\label{rem:Forgetfulmap}
	For \(k\geq2\) let \(f\colon \mathcal{M}_{0,k+1}(A)\to \mathcal{M}_{0,k}(A)\) be the forgetful map that forgets the last point.
	The map \(f\) is \(\C^*\)-equivariant.
\end{rem}

\subsection{Torus action on \texorpdfstring{\(\mathcal{M}_{0,1}(A)\)}{M01(A)}}\label{SSec:TorusActionM01A}
In this section we will construct a torus action on \(\mathcal{M}_{0,1}(A)\).
As in the previous section, the torus action does not descend directly from the torus action on \(\ProjectiveSpace[\C]{1|1}\times \mathcal{M}(A)\).
Instead we have to represent, as before, \(\mathcal{M}_{0,1}(A)\) by a quotient of a smaller supermanifold by a smaller super Lie group.
\begin{lemma}\label{lemma:M01AasQuotient}
	Let \(A\in H_2(N)\) and \(G''\subset\SCP\) be the super Lie subgroup of superconformal automorphisms of \(\ProjectiveSpace[\C]{1|1}\) that preserve the point \(0=[0:1:0]\) of \(\ProjectiveSpace[\C]{1|1}\).
	Then, the moduli space of super \(\targetACI\)-holomorphic curves of genus zero and with one marked point is given by
	\begin{equation}
		\mathcal{M}_{0,1}(A)
		= \faktor{\ProjectiveSpace[\C]{1|1}\times \mathcal{M}(A)}{\SCP}
		= \faktor{\mathcal{M}(A)}{G''}.
	\end{equation}
\end{lemma}
\begin{proof}
	Denoting by \(\gamma\colon G''\to \SCP\) the inclusion and by \(f\colon \mathcal{M}(A)\to \ProjectiveSpace[\C]{1|1}\) the map whose \(B\)-points are given by
	\begin{equation}
		\begin{split}
			\underline{f}(B)\colon \underline{\mathcal{M}(A)}(B) &\to \underline{\ProjectiveSpace[\C]{1|1}\times\mathcal{M}(A)}(B)\\
			\Phi &\mapsto (0, \Phi) \\
		\end{split}
	\end{equation}
	satisfies the conditions of Lemma~\ref{lemma:MoritaEquivalence}.
\end{proof}

The group \(G''\subset\SCP\) preserving the points \(0\) of \(\ProjectiveSpace[\C]{1|1}\) is given by
\begin{equation}
	\underline{G''}(B) = \Set{
		\begin{pmatrix}
			a & c & -a\beta \\
			0 & a^{-1} & 0\\
			0 & \beta & 1
		\end{pmatrix}
		\in \underline{\SCP}(B)
	}.
\end{equation}
Every matrix in \(\underline{G''}(B)\) can be uniquely decomposed into a lift of an element of \(\SGL_\C(2)\) and a supersymmetry:
\begin{equation}
	\begin{pmatrix}
		a & c & -a\beta \\
		0 & a^{-1} & 0\\
		0 & \beta & 1
	\end{pmatrix}
	=
	\begin{pmatrix}
		a & c & 0 \\
		0 & a^{-1} & 0\\
		0 & 0 & 1
	\end{pmatrix}
	\begin{pmatrix}
		1 & 0 & -\beta \\
		0 & 1 & 0\\
		0 & \beta & 1
	\end{pmatrix}
\end{equation}
Furthermore, \(G''\cap \SUSY\) is a normal subgroup of \(G''\) and
\begin{equation}
	G'' = \left(\SGL_\C(2)\cap G''\right)\ltimes \left(\SUSY\cap G''\right).
\end{equation}
By Example~\ref{ex:LiftOfMöbiusTransformations} and Example~\ref{ex:ReflectionOfTheOddDirection}, the action of the lift of \(\SGL_\C(2)\) commutes with the torus action.
The only parameter \(\beta\) enters linearly in the supersymmetry transformations in Example~\ref{ex:SupersymmetryTransformations}.
Hence \(\faktor{\mathcal{M}(A)}{\SUSY\cap G''}\) is split and carries a torus action given by multiplication with \(t\in \C^*\).
\begin{prop}\label{prop:TorusActionM01A}
	The moduli space of super \(\targetACI\)-holomorphic curves of genus zero and with one marked point is isomorphic to
	\begin{equation}
		\mathcal{M}_{0,1}(A) = \faktor{\left(\faktor{\mathcal{M}(A)}{\SUSY\cap G''}\right)}{\SGL_\C(2)\cap G''}.
	\end{equation}
	The quotient \(\faktor{\mathcal{M}(A)}{\SUSY\cap G''}\) is a split supermanifold and carries a torus action which descends to the quotient above.
\end{prop}

\begin{rem}
	The torus action on \(\mathcal{M}(A)\) does not descend to the quotient by \(G''\) directly as the following example shows.
	If we describe a point of \(\mathcal{M}(A)\) by \((\varphi, \psi)\) and consider the action by an element of \(\SUSY\cap G''\) as in Example~\ref{ex:SupersymmetryTransformations} and an element \(t\in \C^*\):
	\begin{equation}\label{eq:phipsiSUSYt}
		(\varphi, \psi)\cdot
		\begin{pmatrix}
			1 & 0 & -\beta \\
			0 & 1 & 0\\
			0 & \beta & 1
		\end{pmatrix}
		\cdot t
		= (\varphi + \beta z_1 \psi_1, \psi + \beta z_1\partial_{z_1}\varphi)\cdot t
		= (\varphi + \beta z_1 \psi_1, t\psi + t\beta z_1\partial_{z_1}\varphi)
	\end{equation}
	whereas
	\begin{equation}\label{eq:phispsitSUSY}
		(\varphi, \psi)\cdot t \cdot
		\begin{pmatrix}
			1 & 0 & -\tilde{\beta} \\
			0 & 1 & 0\\
			0 & \tilde{\beta} & 1
		\end{pmatrix}
		= (\varphi + \tilde{\beta} z_1 t\psi_1, t\psi + \tilde{\beta} z_1\partial_{z_1}\varphi)
	\end{equation}
	In general there is no choice of \(\tilde{\beta}\) such that~\eqref{eq:phipsiSUSYt} and~\eqref{eq:phispsitSUSY} coincide.
\end{rem}

\subsection{Torus action on \texorpdfstring{\(\mathcal{M}_T(\{A_\alpha\})\)}{MT(A)}}\label{SSec:TorusActionMTA}
In this section we argue that the torus actions on \(\mathcal{M}_{0,k}(A)\) extend to the moduli spaces \(\mathcal{M}_T(\Set{A_\alpha})\) of super stable maps with fixed tree type.

Here, as in~\cite{KSY-SQCI}, we follow the notation of~\cite{McDS-JHCST} for trees:
A tree is modeled by a set of vertices \(T\) and a set of edges \(E\subset T\times T\) such that there are no loops.
We will assume that edges are undirected, that is \((\alpha, \beta)\in E\) if and only if \((\beta, \alpha)\) in \(E\).
A \(k\)-marking of the tree is a map \(p\colon \Set{1,\dotsc, k}\to T\).
The tree \(T\) together with the set of homology classes \(\Set{A_\alpha}_{\alpha\in T}\) is said to be stable if \(\#p^{-1}(\alpha)+\#\Set{(\alpha, \beta)\in E}\geq 3\) for all \(\alpha\in T\) such that \(A_\alpha=0\).

\TorusActionMTA{}
\begin{proof}
	Let us first recall the construction of \(\mathcal{M}_T(\Set{A_\alpha})\) from~\cite[Section~4.3]{KSY-SQCI}:
	A super stable map of genus zero over \(B\) modeled on the stable \(k\)-marked tree \(T\) is given by a tuple
	\begin{equation}
		(\bm{z}, \bm{\Phi})
		= \left(\left(\Set{z_{\alpha\beta}}_{E_{\alpha\beta}}, \Set{z_i}_{1\leq i\leq k}\right), \Set{\Phi_\alpha}_{\alpha\in T}\right)
		\in {\left(\underline{\ProjectiveSpace[\C]{1|1}}(B)\right)}^{2\#E+k}\times \prod_{\alpha\in T}\underline{\mathcal{M}(A_\alpha)}(B),
	\end{equation}
	where \(z_{\alpha\beta}\) represent the nodal point connecting the copy of \(\ProjectiveSpace[\C]{1|1}\) at \(\alpha\in T\) with the copy of \(\ProjectiveSpace[\C]{1|1}\) at \(\beta\in T\), \(z_i\) is a marked point at the vertex \(p(i)=\alpha\) and \(\Phi_\alpha\colon \ProjectiveSpace[\C]{1|1}\to N\) is a super \(\targetACI\)-holomorphic curve.
	The tuple \((\bm{z}, \bm{\Phi})\) needs to satisfy the following:
	\begin{itemize}
		\item
			The reduction of the nodal points \(z_{\alpha\beta}\) and the marked points \(z_i\) with \(p(i)=\alpha\) on the node \(\alpha\) need to be distinct.
			We denote by \(Z^T\subset{\left(\ProjectiveSpace[\C]{1|1}\right)}^{2\#E+k}\) the subsupermanifold of tuples \(\bm{z}\) satisfying this condition.
		\item
			At all nodes \(\alpha\), \(\beta\in T\) super stable maps have to satisfy \(\Phi_\alpha\circ z_{\alpha\beta} = \Phi_\beta\circ z_{\beta\alpha}\).
			That is super stable maps over \(B\) need to lie in the edge diagonal
			\begin{equation}
				\Delta^T
				=\Set{\left(n_{\alpha\beta}\right)\in N^{2\#E}\given n_{\alpha\beta}=n_{\beta\alpha}}
			\end{equation}
			under the edge evaluation map defined by
			\begin{equation}
				\begin{split}
					\underline{\ev^T}(B)\colon \underline{Z^T}(B)\times\underline{M^T}(B) &\to \underline{N^{2\#E}}(B) \\
					(\bm{z}, \bm{\Phi}) &\mapsto {\left(\Phi_\alpha\circ z_{\alpha\beta}\right)}_{E_{\alpha\beta}}
				\end{split}
			\end{equation}
			Here we set \(M^T=\prod_{\alpha\in T}\mathcal{M}(A)\).
			By the assumption that the edge evaluation map \(\ev^T\) is transversal to \(\Delta^T\), we have that \({\left(\ev^T\right)}^{-1}(\Delta^T)=\left(Z^T\times M^T\right)\times_{N^{2\#E}} \Delta^T\) is a supermanifold.
			As \(\ev^T\) is a map to a classical manifold with only even dimensions, the supermanifold \({\left(\ev^T\right)}^{-1}(\Delta^T)\) is the restriction of \(Z^T\times M^T\) to the open subset \(\underline{{\left(\ev^T\right)}^{-1}(\Delta^T)}(\R^{0|0})\subset\underline{Z^T\times M^T}(\R^{0|0})\).
	\end{itemize}

	A \(B\)-point \(\bm{g}=\Set{g_\alpha}_{\alpha\in T}\) of the automorphism group \(G^T=\prod_{\alpha\in T}\SCP\) of nodal curves modeled on \(T\) acts on super stable maps modeled on \(T\) by
	\begin{equation}
		\bm{g}\cdot\left(\bm{z}, \bm{\Phi}\right)
		= \left(\left(\Set{g_\alpha(z_{\alpha\beta})}, \Set{g_{p(i)}(z_i)}\right), \Set{\Phi_\alpha\circ g_\alpha^{-1}}\right).
	\end{equation}
	The edge evaluation map \(\ev^T\) is \(G^T\) invariant and hence \(G^T\) acts on \({\left(\ev^T\right)}^{-1}(\Delta^T)\).
	The moduli space of simple super stable maps of genus zero, modeled on \(T\) and with partition~\(\Set{A_\alpha}\) of the homology classes is then obtained in~\cite[Theorem~4.3.1]{KSY-SQCI} as a superorbifold by
	\begin{equation}
		\mathcal{M}_T(\Set{A_\alpha})
		= \faktor{{\left(\ev^T\right)}^{-1}(\Delta^T)}{G^T}.
	\end{equation}
	The group \(G^T=\prod_{\alpha\in T}\SCP\) acts factorwise on the nodes, marked points and super \(\targetACI\)-holomorphic maps indexed by \(\alpha\in T\).
	Hence
	\begin{equation}
		\mathcal{M}_T(\Set{A})\subset \prod_{\alpha\in T} \mathcal{M}_{0,k_\alpha+\#E_\alpha} (A_\alpha).
	\end{equation}
	Here, the node \(\alpha\in T\) has \(\#E_\alpha\) many outgoing edges and \(k_\alpha=p^{-1}(\alpha)\) many marked points.

	We have argued in Proposition~\ref{prop:TorusActionM0k}, Proposition~\ref{prop:TorusActionM0kA} and Proposition~\ref{prop:TorusActionM01A} how every moduli space \(\mathcal{M}_{0,k_\alpha+\#E_\alpha} (A_\alpha)\) obtains a torus action via splitness.
	The restriction to the open submanifold \(\underline{{\left(\ev^T\right)}^{-1}(\Delta^T)}(\R^{0|0})\) is torus invariant.
	Hence we obtain an induced torus action on \(\mathcal{M}_T(\Set{A_\alpha})\).
\end{proof}

By construction the moduli space \(\mathcal{M}_T(\Set{A_\alpha})\) is a superorbifold with isotropy \(\Z_2^{\#T}\) at all \(\R^{0|0}\)-points.
Its reduced space, forgetting the \(\Z_2^{\#T}\)-action, is the moduli space \(M_T(\Set{A_\alpha})\) of classical stable maps modeled on \(T\).
The manifold \(M_T(\Set{A_\alpha})\) carries a normal bundle \(N_{T,\Set{A_\alpha}}=\bigoplus_{\alpha\in T}p_\alpha^*N_{k_\alpha+\#E_\alpha, A_\alpha}\) with \(\C^*\)-action such that
\begin{equation}
	\underline{\mathcal{M}_T(\Set{A_\alpha})}(\R^{0|1})
	= \faktor{N_{T,\Set{A_\alpha}}}{\Z_2^{\#T}},
\end{equation}
where each copy of \(\Z_2\) acts on one summand of the normal bundle.

\begin{rem}
	Transversality of the edge evaluation map~\(\ev^T\) can be achieved by generic variation of the almost complex structure \(\targetACI\).
	However, a generic perturbation of \(\targetACI\) will usually not be Kähler.
	Transversality of the edge evaluation map is guaranteed if \(N\) is Kähler and its automorphism group acts transitively by holomorphic isomorphisms, see~\cite[Proposition~7.4.3]{McDS-JHCST}.
	The conditions are in particular satisfied for \(N=\ProjectiveSpace[\C]{n}\).
\end{rem}

\begin{rem}
	Let \(T'\) be the tree that arises from \(T\) by forgetting the \(k\)-th marked point and possibly stabilization.
	It follows from Remark~\ref{rem:Forgetfulmap} that the map \(f\colon \mathcal{M}_T(\Set{A_\alpha})\to \mathcal{M}_{T'}(\Set{A_\alpha})\) is \(\C^*\)-equivariant.

	The Theorem~\ref{thm:TorusActionMTA} also constructs a torus action on the moduli spaces \(\mathcal{M}_T\) of super stable curves modeled on the stable tree \(T\).
	The forgetful map \(\pi\colon \mathcal{M}_T(\Set{A_\alpha})\to \mathcal{M}_T\) is \(\C^*\)-equivariant.
\end{rem}

\printbibliography

\textsc{Enno Keßler\\
Max-Planck-Institut für Mathematik,
Vivatsgasse 7,
53111 Bonn,
Germany}\\
\texttt{kessler@mpim-bonn.mpg.de}\\

\textsc{Artan Sheshmani\\
Harvard University,
Jefferson Laboratory,
17 Oxford Street,
Cambridge, MA 02138,
USA\\[.5em]
Massachusetts Institute of Technology (MIT),
IAiFi Institute,
182 Memorial Drive,
Cambridge, MA 02139,
USA}\\
\texttt{artan@mit.edu}\\

\textsc{Shing-Tung Yau\\
Yau Mathematical Sciences Center,
Tsinghua University,
Haidian District,
Beijing,
China
}\\
\texttt{styau@tsinghua.edu.cn}

\end{document}